\documentclass[oneside,12pt]{amsart}

\usepackage[a4paper,left=30mm, right=30mm, top=30mm, bottom=30mm]{geometry}

\usepackage{amsmath,amsthm,amsfonts,amssymb}
\usepackage[all]{xy}
\usepackage{color}
\usepackage{xparse}
\usepackage{aliascnt}
\usepackage[pagebackref,colorlinks,citecolor=blue,linkcolor=blue,urlcolor=blue,filecolor=blue]{hyperref}
\usepackage{bbm}
\usepackage{enumerate}
\usepackage{verbatim}
\usepackage[normalem]{ulem} 

\usepackage{tikz}
\usetikzlibrary{shapes.geometric}
\tikzset{
v/.style={draw, fill, circle, minimum size=1.5mm, inner sep=0},
b/.style={draw , regular polygon,regular polygon sides=4, minimum size=1.5mm, inner sep=.5mm},
e/.style={very thick},
vs/.style={draw, fill, circle, minimum size=1mm, inner sep=0},
bs/.style={draw,  regular polygon,regular polygon sides=4, minimum size=2mm, inner sep=0mm},
es/.style={thick}
}
\newlength{\nodeheight}
\newlength{\nodewidth}
\usetikzlibrary{arrows,matrix,positioning}

\numberwithin{thmcounter}{section}
\newaliascnt{thmauto}{thmcounter}

\newaliascnt{Defauto}{thmcounter}

\newaliascnt{exauto}{thmcounter}

\newaliascnt{lemauto}{thmcounter}

\newaliascnt{propauto}{thmcounter}

\newaliascnt{corauto}{thmcounter}

\newaliascnt{remauto}{thmcounter}

\newaliascnt{clmauto}{thmcounter}

\newtheorem{atheorem}{Theorem}
\newtheorem{acor}[atheorem]{Corollary}

\newtheorem{thm}[thmauto]{Theorem}
\newtheorem{lem}[lemauto]{Lemma}
\newtheorem{prop}[propauto]{Proposition}

\theoremstyle{definition}
\newtheorem{Def}[Defauto]{Definition}
\newtheorem{defn}[Defauto]{Definition}
\newtheorem{ex}[exauto]{Example}

\numberwithin{equation}{section}

\DeclareMathOperator{\Tor}{Tor}
\DeclareMathOperator{\Tot}{Tot}

\newcommand{\IE}{{}^{I}E}
\newcommand{\IIE}{{}^{II}E}

\newcommand{\twoheadlongrightarrow}{\relbar\joinrel\twoheadrightarrow}

\DeclareMathOperator{\Br}{Br}

\renewcommand{\P}{\ensuremath\operatorname{P}}

\providecommand{\fS}{\ensuremath\mathfrak{S}}

\providecommand{\tens}[1][{}]{\otimes_{#1}}

\renewcommand{\t}{\mathbbm{1}}

\newcommand{\calA}{\mathcal{A}}
\newcommand{\calB}{\mathcal{B}}
\newcommand{\calM}{\mathcal{M}}

\title{The homology of the partition algebras}
\author{Rachael Boyd}
\address{School of Mathematics and Statistics, University of Glasgow}
\email{rachael.boyd@glasgow.ac.uk}
\author{Richard Hepworth}
\address{Institute of Mathematics, University of Aberdeen}
\thanks{For the purpose of open access, the authors have applied a Creative Commons Attribution (CC BY) licence to any Author Accepted Manuscript version arising from this submission.}
\email{r.hepworth@abdn.ac.uk}
\author{Peter Patzt}
\address{Department of Mathematics, University of Oklahoma}
\email{ppatzt@ou.edu}

 \subjclass[2020]{
        20J06, 
        16E40 
        (primary),
        20B30  	
        (secondary)
    }
    \keywords{Homology, homological stability, partition algebras}

\begin{document}
\maketitle

\begin{abstract}
    We show that the homology of the partition algebras, interpreted as appropriate $\Tor$-groups, is isomorphic to that of the symmetric groups in a range of degrees that increases with the number of nodes.
    Furthermore, we show that when the defining parameter~$\delta$ of the partition algebra is invertible, the homology of the partition algebra is in fact isomorphic to the homology of the symmetric group in all degrees.
    These results parallel those obtained for the Brauer algebras in the authors' earlier work, but with significant differences and difficulties in the \emph{inductive resolution} and \emph{high acyclicity} arguments required to prove them.
    Our results join the growing literature on homological stability for algebras, which now encompasses the Temperley-Lieb, Brauer and partition algebras, as well as the Iwahori-Hecke algebras of types $A$ and $B$.
\end{abstract}

\section{Introduction}

In the last few years it has become increasingly apparent that the techniques of \emph{homological stability}, which are most commonly applied to families of groups, can be successfully applied to families of \emph{algebras}, where homology is interpreted as an appropriate $\Tor$ group.
Indeed, the papers of Boyd--Hepworth~\cite{BoydHepworthStability}, Boyd--Hepworth--Patzt~\cite{BHP2021}, Hepworth~\cite{HepworthIH} and Moselle~\cite{Moselle} proved homological stability for Temperley-Lieb algebras, Brauer algebras, and Iwahori-Hecke algebras of types $A$ and $B$ respectively, and identified the stable homology in the first two cases.
The Temperley-Lieb and Brauer algebras failed to satisfy a certain flatness condition that holds automatically for families of groups, necessitating the introduction of the new technique of \emph{inductive resolutions}.
Using related techniques, Sroka~\cite{Sroka} showed that the homology of the Temperley-Lieb algebra on an odd number of strands vanishes in positive degrees, in contrast to the known non-vanishing for an even number of strands.
More recently, Boyde~\cite{Boyde} used a careful study of idempotents to unify and generalise the `invertible parameter' results from~\cite{BoydHepworthStability,BHP2021}, together with Sroka's vanishing result.
In this paper, we prove homological stability for the partition algebras, and we identify their stable homology.

The \emph{partition algebras} were introduced independently by Jones~\cite{Jones_Potts} and Martin~\cite{Martin_JKTR} for their relevance in studying Potts models in statistical mechanics.
They are also important in representation theory as a Schur--Weyl dual to the symmetric group, as in the work of Halverson--Ram~\cite[Theorems 5.4, 3.6]{HalversonRam} and Bowman--Doty--Martin~\cite{BowmanDotyMartin}.
They contain a rich variety of subalgebras, including the planar partition, rook Brauer, rook, planar rook, Brauer, Motzkin and Temperley-Lieb algebras. 

Given a commutative ring $R$, an element $\delta\in R$, and a non-negative integer~$n$, the partition algebra~$\P_n(R,\delta)$ is defined to be the free module over~$R$ with basis given by the partitions of the set~$\{-1,\dots,-n,1,\dots, n\}$.
These partitions can be drawn as diagrams with~$n$ nodes labelled~$-1,\dots, -n$ on the left and~$n$ nodes labelled~$1,\dots, n$ on the right. 
Nodes in the same block of a partition are then joined by edges. For ease of drawing, we do not include all edges but instead rely on transitivity. 
Disconnected nodes are allowed, corresponding to blocks of size one. 
For example, the following diagram shows the basis element $\{\{-1,-3\},\{-2,-4,4\},\{2,3\},\{1\}\}$ of $\P_4(R,\delta)$.
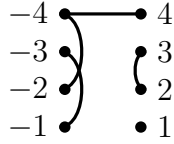
\[
    \begin{tikzpicture}
    \fill (0,0)           circle (.75mm) node[left=2pt](a5) {$-4$};
    \fill (0,\nodeheight) circle (.75mm) node[left=2pt](a4) {$-3$};
    \multiply\nodeheight by 2
    \fill (0,\nodeheight) circle (.75mm) node[left=2pt](a3) {$-2$};
    \divide\nodeheight by 2
    \multiply\nodeheight by 3
    \fill (0,\nodeheight) circle (.75mm) node[left=2pt](a2) {$-1$};
    \divide\nodeheight by 3
    
    \fill (\nodewidth,0) circle (.75mm) node[right=2pt](b5) {$4$};
    \fill (\nodewidth,\nodeheight) circle (.75mm) node[right=2pt](b4) {$3$};
    \multiply\nodeheight by 2
    \fill (\nodewidth,\nodeheight) circle (.75mm) node[right=2pt](b3) {$2$};
    \divide\nodeheight by 2
    \multiply\nodeheight by 3
    \fill (\nodewidth,\nodeheight)  circle (.75mm) node[right=2pt](b2) {$1$};
    \divide\nodeheight by 3
    
    \draw[e] (a2) to[out=0, in=0]   (a4);
    \draw[e] (b3) to[out=180, in=180] (b4);
    \draw[e] (a5) to[out=0, in=0] (a3);
    \draw[e] (a5) to[out=0, in=180] (b5);
\end{tikzpicture}
\]
Multiplication is given by placing the diagrams side by side, identifying the middle nodes, and replacing any blocks not connected to the right or left by a factor of~$\delta$. 

Diagrams in which every node on the left is connected to a single node on the right, and nothing else, are called \emph{permutation diagrams}, and are in bijection with elements of the symmetric group $\fS_n$.
This gives rise to inclusion and projection maps
 \[
        R\fS_n\xrightarrow{\ \iota\ }\P_n(R,\delta)
       \xrightarrow{\ \pi\ }R\fS_n
    \]
where $\iota$ sends permutations to permutation diagrams, and $\pi$ does the reverse while sending all remaining diagrams to $0$.
In particular, $\pi\circ \iota$ is the identity map on $R\fS_n$.

We denote the trivial module of $R\fS_n$ by $\t$. 
Pulling back along $\pi$, we obtain the \emph{trivial module} $\t$ of $\P_n(R,\delta)$.
This gives us the homology groups $H_\ast(\fS_n,\t)=\Tor_\ast^{R\fS_n}(\t,\t)$ of $\fS_n$ and $\Tor_\ast^{\P_n(R,\delta)}(\t,\t)$ of $\P_n(R,\delta)$.
There are induced homomorphisms $\iota_\ast$ and $\pi_\ast$ on homology groups for which  $\pi_\ast\circ\iota_\ast$ is again the identity, so that the homology of $\fS_n$ appears as a direct summand of the homology of $\P_n(R,\delta)$.

\begin{atheorem}\label{thma}
     Suppose that $\delta$ is invertible in $R$.
    Then the homology of the partition algebra is isomorphic to the homology of the symmetric group:
    \[
        \Tor_\ast^{\P_n(R,\delta)}(\t,\t)\cong H_\ast(\fS_n;\t)
    \]
    Indeed, the inclusion and projection maps
    \[
        R\fS_n\xrightarrow{\ \iota\ }\P_n(R,\delta)
       \xrightarrow{\ \pi\ }R\fS_n
    \]
    induce inverse isomorphisms
    \[
        \Tor_\ast^{R\fS_n}(\t,\t)
        \xrightarrow[\cong]{\ \iota_\ast\ }
        \Tor_\ast^{\P_n(R,\delta)}(\t,\t)
        \xrightarrow[\cong]{\ \pi_\ast \ }
        \Tor_\ast^{R\fS_n}(\t,\t).
    \]
\end{atheorem}

Our second result holds without any assumptions on the value of $\delta$.

\begin{atheorem}\label{thmb}
The inclusion map $\iota\colon R\mathfrak S_n \to \P_n(R,\delta)$ induces a map in homology
\[ \iota_\ast\colon H_i(\mathfrak S_n;\t) \longrightarrow \Tor^{\P_n(R,\delta)}_i(\t,\t)\]
that is an isomorphism in the range $n \ge 2i+1$.
\end{atheorem}

An immediate consequence of \autoref{thmb} is the following corollary.

\begin{acor}
    The partition algebras satisfy homological stability, that is, the inclusion $\P_{n-1}(R,\delta)\hookrightarrow \P_n(R,\delta)$ induces a map
	\[
	\Tor^{\P_{n-1}(R,\delta)}_i(\t,\t)
	\longrightarrow
	\Tor^{\P_{n}(R,\delta)}_i(\t,\t)
	\]
    that is an isomorphism in degrees $n\geq 2i+1$, and this stable range is sharp.
    Furthermore, $\P_n(R,\delta)$ and $\fS_n$ have the same stable homology:
	\[
	\lim_{n\to\infty}H_\ast(\fS_n;\t)
        \cong
	\lim_{n\to\infty}\Tor_\ast^{\P_n(R,\delta)}(\t,\t).
	\]
\end{acor}
	
The first part of this corollary follows by combining \autoref{thmb} with Nakaoka's homological stability result for the symmetric groups, for which the stable range is sharp~(\cite{N}).
For the stable homology, the left hand side of this isomorphism is well known by the Barratt--Priddy--Quillen theorem~\cite{BarrattPriddy,QuillenQ}. 
The above results exactly parallel the situation for the Brauer algebras, and as discussed in \cite{BHP2021} are reminiscent of the relationship between $\fS_n$ and the automorphism groups of free groups $\mathrm{Aut}(F_n)$ (see Galatius \cite{Galatius}).

\subsection{Outline, and comparison to previous work}
In \autoref{sec-partition algebras} we introduce partition algebras and provide the necessary background needed for the rest of the paper. 
In \autoref{sec-inductive-principle} we restate an abstract form of the principle that lies behind the technique of inductive resolutions that was introduced in~\cite{BoydHepworthStability}, and was a crucial ingredient in \cite{BoydHepworthStability} and~\cite{BHP2021}.
In \autoref{sec-inductive resolutions} we establish the existence of inductive resolutions for the partition algebras. These are significantly more complicated than the Temperley-Lieb~\cite{BoydHepworthStability} and Brauer~\cite{BHP2021} cases, and we find that we must consider several families of distinct modules in order to carry out our induction argument. 
In \autoref{sec-replacing shapiros lemma} we follow the argument of~\cite{BHP2021} to replace Shapiro's lemma in the setting of partition algebras. 
The high connectivity result required for any new proof of homological stability is found in \autoref{sec-high connectivity}.
Like our inductive resolutions argument, this is again more complicated than the analogous result in~\cite{BHP2021}, and heavily utilises the high connectivity of the complex of injective words with separators, introduced in that paper. 
We finish in \autoref{sec-proof of thmb} by giving an account of the proof of the main theorem, which follows the same general argument as in \cite{BHP2021}.

It is common, in homological stability for families of groups, to find that proofs of different results have a very similar overall structure, yet the proofs that the relevant complexes are highly acyclic can differ radically.
What we can now see in homological stability for algebras, comparing the work of this paper to that of~\cite{BoydHepworthStability} and~\cite{BHP2021}, is an analogous situation where the overall technique is used in multiple situations, but the details of the acyclicity proofs --- and now also of the inductive resolutions proofs --- are where the important differences and difficulties lie.

\subsection{Acknowledgements}
The first author was supported by EPSRC Postdoctoral Fellowship EP/V043323/1 and EP/V043323/2.
The third author was supported by a Simons Collaboration Grant.
Our thanks to the anonymous referee for their detailed reading and helpful comments.

\section{Partition algebras}\label{sec-partition algebras}

In this section we introduce the partition algebra, together with some specific elements and modules that will be important later in the paper.

\begin{defn}[The partition algebra \cite{Jones_Potts,Martin_JKTR}] 
    As explained in the introduction, if~$R$ is a commutative ring,~$\delta$ is a chosen element in~$R$, and~$n$ is a non-negative integer, then the partition algebra~$\P_n(R,\delta)$ is defined to be the free module over~$R$ with basis given by the partitions of the set~$\{-n,\dots,-1,1,\dots, n\}$.
    These are drawn as diagrams with nodes~$-1,\dots, -n$ on the left and nodes~$1,\dots, n$ on the right, with arcs indicating which nodes lie in the same block of the partition.
    (We allow ourselves to omit some arcs and instead use transitivity to determine the blocks.)
    An example is shown in \autoref{fig:P5ex}. 
    Multiplication is given by placing the diagrams side by side, identifying the middle nodes, and replacing any blocks not connected to the right or left by a factor of~$\delta$, as in~\autoref{fig:MultP5}. 
\end{defn}

We will use the terms `partition' and `diagram' interchangeably to mean a basis element of $\P_n(R,\delta)$, and we will frequently abbreviate $\P_n(R,\delta)$ as $\P_n$ .

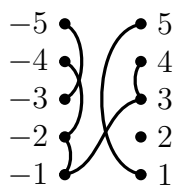
\begin{figure}[h!]
\centering
\begin{tikzpicture}
\fill (0,0)           circle (.75mm) node[left=2pt](a5) {$-5$};
\fill (0,\nodeheight) circle (.75mm) node[left=2pt](a4) {$-4$};
\multiply\nodeheight by 2
\fill (0,\nodeheight) circle (.75mm) node[left=2pt](a3) {$-3$};
\divide\nodeheight by 2
\multiply\nodeheight by 3
\fill (0,\nodeheight) circle (.75mm) node[left=2pt](a2) {$-2$};
\divide\nodeheight by 3
\multiply\nodeheight by 4
\fill (0,\nodeheight) circle (.75mm) node[left=2pt](a1) {$-1$};
\divide\nodeheight by 4

\fill (\nodewidth,0) circle (.75mm) node[right=2pt](b5) {$5$};
\fill (\nodewidth,\nodeheight) circle (.75mm) node[right=2pt](b4) {$4$};
\multiply\nodeheight by 2
\fill (\nodewidth,\nodeheight) circle (.75mm) node[right=2pt](b3) {$3$};
\divide\nodeheight by 2
\multiply\nodeheight by 3
\fill (\nodewidth,\nodeheight)  circle (.75mm) node[right=2pt](b2) {$2$};
\divide\nodeheight by 3
\multiply\nodeheight by 4
\fill (\nodewidth,\nodeheight)  circle (.75mm) node[right=2pt](b1) {$1$};
\divide\nodeheight by 5

\draw[e] (a1) to[out=0, in=180] (b3);
\draw[e] (a1) to[out=0, in=0] (a2);
\draw[e] (a2) to[out=0, in=0]   (a4);
\draw[e] (b3) to[out=180, in=180] (b4);
\draw[e] (a5) to[out=0, in=0] (a3);
\draw[e] (b1) to[out=180,in=180] (b5);

\end{tikzpicture}
\caption{Visualization of the partition $\{\{-5,-3\},\{-4,-2,-1,3,4\},\{1,5\},\{2\}\}$}
\label{fig:P5ex}
\end{figure}

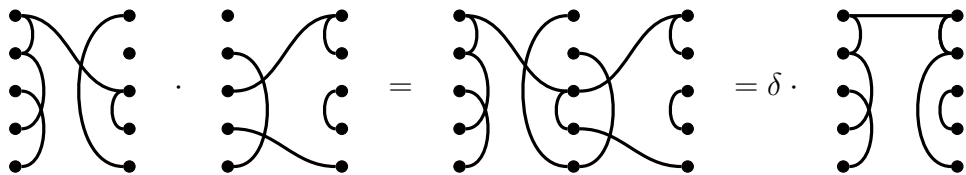
\begin{figure}[h!]
\centering
\[
\begin{tikzpicture}[x=1.5cm,y=-.5cm,baseline=-1.05cm]

\node[v] (a1) at (0,0) {};
\node[v] (a2) at (0,1) {};
\node[v] (a3) at (0,2) {};
\node[v] (a4) at (0,3) {};
\node[v] (a5) at (0,4) {};

\node[v] (b1) at (1,0) {};
\node[v] (b2) at (1,1) {};
\node[v] (b3) at (1,2) {};
\node[v] (b4) at (1,3) {};
\node[v] (b5) at (1,4) {};

\draw[e] (a1) to[out=0, in=180] (b3);
\draw[e] (a1) to[out=0, in=0] (a2);
\draw[e] (a2) to[out=0, in=0]   (a4);
\draw[e] (b3) to[out=180, in=180] (b4);
\draw[e] (a5) to[out=0, in=0] (a3);
\draw[e] (b1) to[out=180,in=180] (b5);

\end{tikzpicture}
\quad
\cdot
\quad
\begin{tikzpicture}[x=1.5cm,y=-.5cm,baseline=-1.05cm]

\node[v] (b1) at (0,0) {};
\node[v] (b2) at (0,1) {};
\node[v] (b3) at (0,2) {};
\node[v] (b4) at (0,3) {};
\node[v] (b5) at (0,4) {};

\node[v] (c1) at (1,0) {};
\node[v] (c2) at (1,1) {};
\node[v] (c3) at (1,2) {};
\node[v] (c4) at (1,3) {};
\node[v] (c5) at (1,4) {};

\draw[e] (b2) to[out=0,in=0] (b5);
\draw[e] (b3) to[out=0,in=180] (c1);
\draw[e] (b4) to[out=0,in=180] (c5);
\draw[e] (c1) to[out=180,in=180] (c2);
\draw[e] (c3) to[out=180,in=180] (c4);

\end{tikzpicture}
\quad
=
\quad
\begin{tikzpicture}[x=1.5cm,y=-.5cm,baseline=-1.05cm]

\node[v] (a1) at (0,0) {};
\node[v] (a2) at (0,1) {};
\node[v] (a3) at (0,2) {};
\node[v] (a4) at (0,3) {};
\node[v] (a5) at (0,4) {};

\node[v] (b1) at (1,0) {};
\node[v] (b2) at (1,1) {};
\node[v] (b3) at (1,2) {};
\node[v] (b4) at (1,3) {};
\node[v] (b5) at (1,4) {};

\node[v] (c1) at (2,0) {};
\node[v] (c2) at (2,1) {};
\node[v] (c3) at (2,2) {};
\node[v] (c4) at (2,3) {};
\node[v] (c5) at (2,4) {};

\draw[e] (a1) to[out=0, in=180] (b3);
\draw[e] (a1) to[out=0, in=0] (a2);
\draw[e] (a2) to[out=0, in=0]   (a4);
\draw[e] (b3) to[out=180, in=180] (b4);
\draw[e] (a5) to[out=0, in=0] (a3);
\draw[e] (b1) to[out=180,in=180] (b5);

\draw[e] (b2) to[out=0,in=0] (b5);
\draw[e] (b3) to[out=0,in=180] (c1);
\draw[e] (b4) to[out=0,in=180] (c5);
\draw[e] (c1) to[out=180,in=180] (c2);
\draw[e] (c3) to[out=180,in=180] (c4);

\end{tikzpicture}
\quad
= \delta \cdot
\quad
\begin{tikzpicture}[x=1.5cm,y=-.5cm,baseline=-1.05cm]

\node[v] (a1) at (0,0) {};
\node[v] (a2) at (0,1) {};
\node[v] (a3) at (0,2) {};
\node[v] (a4) at (0,3) {};
\node[v] (a5) at (0,4) {};

\node[v] (c1) at (1,0) {};
\node[v] (c2) at (1,1) {};
\node[v] (c3) at (1,2) {};
\node[v] (c4) at (1,3) {};
\node[v] (c5) at (1,4) {};

\draw[e] (a1) to[out=0, in=180] (c1);
\draw[e] (a1) to[out=0, in=0] (a2);
\draw[e] (a2) to[out=0, in=0]   (a4);
\draw[e] (a5) to[out=0, in=0] (a3);
\draw[e] (c1) to[out=180,in=180] (c2);
\draw[e] (c1) to[out=180,in=180] (c2);
\draw[e] (c3) to[out=180,in=180] (c4);
\draw[e] (c2) to[out=180,in=180] (c5);
\end{tikzpicture}
\]
\caption{Multiplication in the partition algebra}
\label{fig:MultP5}
\end{figure}

The partition algebra is generated by three types of diagrams \cite{M96}, corresponding to the following partitions:
\begin{itemize}
    \item For~$1\leq i\leq n-1$, $S_i$ is the diagram corresponding to the partition with blocks of pairs~$\{-j,j\}$ for~$j\neq i, i+1$, together with~$\{-(i+1),i\}$ and $\{-i,(i+1)\}$. These generate the group ring of the symmetric group, $\fS_n$, as a subalgebra of~$\P_n$.
    \item For~$1\leq i\neq j\leq n-1$, $V_{ij}$ is the diagram corresponding to the partition with blocks of pairs~$\{-k,k\}$ for~$k\neq i, j$ and one block of size four~$\{-j,-i,i,j\}$.
    \item For~$1\leq i\leq n$, $T_i$ is the diagram corresponding to the partition with blocks of pairs~$\{-j,j\}$ for~$j\neq i$ and two singleton blocks~$\{-i\}$ and~$\{i\}$.
\end{itemize}
See \autoref{fig-Pn-elements} for depictions of some of these.
\begin{figure}[h!]
    \centering
    \[
    \begin{tikzpicture}[x=1.5cm,y=-.5cm,baseline=-1.05cm]
    
    \node[v] (a1) at (0,0) {};
    \node[v] (a2) at (0,1) {};
    \node[v] (a3) at (0,2) {};
    \node[v] (a4) at (0,3) {};
    
    \node[v] (b1) at (1,0) {};
    \node[v] (b2) at (1,1) {};
    \node[v] (b3) at (1,2) {};
    \node[v] (b4) at (1,3) {};
    
    \draw[e] (a1) to[out=0, in=180] (b1);
    \draw[e] (a2) to[out=0, in=180] (b3);
    \draw[e] (a3) to[out=0, in=180]   (b2);
    \draw[e] (a4) to[out=0, in=180] (b4);

    \node at (0.5,4) {$S_2$};
    \end{tikzpicture}
    \qquad \qquad
    \begin{tikzpicture}[x=1.5cm,y=-.5cm,baseline=-1.05cm]
    
    \node[v] (a1) at (0,0) {};
    \node[v] (a2) at (0,1) {};
    \node[v] (a3) at (0,2) {};
    \node[v] (a4) at (0,3) {};
    
    \node[v] (b1) at (1,0) {};
    \node[v] (b2) at (1,1) {};
    \node[v] (b3) at (1,2) {};
    \node[v] (b4) at (1,3) {};
   
    \draw[e] (a1) to[out=0, in=180] (b1);
    \draw[e] (a2) to[out=0, in=180] (b2);
    \draw[e] (a3) to[out=0, in=180] (b3);
    \draw[e] (a4) to[out=0, in=180] (b4);
    \draw[e] (a2) to[out=0, in=0] (a4);
    \draw[e] (b2) to[out=180, in=180] (b4);
    
    \node at (0.5,4) {$V_{13}$};
    \end{tikzpicture}
    \qquad\qquad
    \begin{tikzpicture}[x=1.5cm,y=-.5cm,baseline=-1.05cm]
    
    \node[v] (a1) at (0,0) {};
    \node[v] (a2) at (0,1) {};
    \node[v] (a3) at (0,2) {};
    \node[v] (a4) at (0,3) {};
    
    \node[v] (b1) at (1,0) {};
    \node[v] (b2) at (1,1) {};
    \node[v] (b3) at (1,2) {};
    \node[v] (b4) at (1,3) {};
   
    \draw[e] (a1) to[out=0, in=180] (b1);
    \draw[e] (a3) to[out=0, in=180] (b3);
    \draw[e] (a4) to[out=0, in=180] (b4);
    
    \node at (0.5,4) {$T_3$};
    \end{tikzpicture}
    \]
    \caption{The elements $S_2,V_{13},T_3\in\P_4$}
    \label{fig-Pn-elements}
\end{figure}
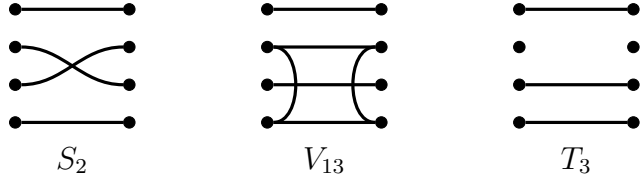

We now introduce the modules we will be working with.

Recall that by a \emph{permutation diagram} we mean a diagram in which each node on the left is joined to a single node on the right, and nothing else.
Equivalently, permutation diagrams are ones that do not contain any singletons on the right \emph{or} any blocks that contain $\geq 2$ elements on the right.

\begin{defn}[The trivial module $\t$]
    For any~$n$, we define the \emph{trivial}~$R\fS_n$-bimodule~$\t$ to be the module given by the ring~$R$, upon which the permutations act as the identity.

    For any~$n$, we define the \emph{trivial}~$\P_n$-bimodule~$\t$ to be the module given by the ring~$R$, upon which the permutation diagrams act as the identity, and all other diagrams act as $0$. This is the same as acting with $\P_n$ on $R$ via the projection $\pi \colon \P_n \to \fS_n$.
\end{defn}

\begin{defn}
    For~$m\leq n$, we can view~$\P_m$ as a subalgebra of~$\P_n$. Given a partition of~$\{-m,\dots,-1,1,\dots ,m\}$,
    the map which sends~$(\pm 1,\dots, \pm m)$ to~$(\pm (n-m+1), \dots , \pm n)$ induces a partition on $\{-n\dots, -(n-m+1), (n-m+1), \dots , n\}$. We add the blocks $\{-i,i\}$ for all $i \in \{1, \dots, (n-m)\}$, resulting in a partition in $\P_n$.
    Pictorially, we are taking diagrams in $\P_m$ and extending them to ones in $\P_n$ by adding new nodes \emph{below} the existing ones, with horizontal connections between the new nodes. 
    Then, under the action of this subalgebra,~$\P_n$ can be viewed as a left $\P_n$-module and a right~$\P_m$-module, and we obtain the induced left $\P_n$-module~$\P_n\otimes_{\P_m}\t$.
\end{defn}

The following proposition is taken from~\cite{PatztRepstab}.

\begin{prop}[{\cite[Proposition 2.5]{PatztRepstab}}]\label{prop:HomBr}
The induced module $\P_n \otimes_{\P_{m}} \t$ is a free $R$-module and a quotient of~$\P_n$. 

In terms of diagrams, a basis for this module is the set of diagrams in which the top $m$ nodes on the right are placed under a box, satisfying the following condition: 
\begin{itemize}
    \item The box is connected to exactly~$m$ distinct blocks.
\end{itemize}
Under this description, the action of $\P_n$ on $\P_n\otimes_{\P_m}\t$ is given by pasting and simplifying the diagrams just as in the multiplication of $\P_n$, and then identifying a diagram with $0$ if it violates the condition above.
\end{prop}

Thus there are two ways that a diagram could be identified with $0$ after left multiplication by a diagram in~$\P_n$: 
One of the blocks attached to the box could, after pasting, consist only of nodes in the centre (visually, that block is free to be retracted into the box, and then disappears).
Alternatively, two or more distinct blocks that were attached to the box can become fused into a single block (visually, there is now a path of arcs with both ends attached to the box).
These two possibilities correspond to the two ways in which a diagram in $\P_m$ can fail to be a permutation diagram, and therefore act as $0$ on $\t$: It can have a singleton on the right, or it can have two nodes on the right belonging to the same block.

\begin{ex}
    Figure~\ref{fig:MultInd2} depicts the module structure of $\P_5\otimes_{\P_3}\t$.
    In the first example one of the blocks connected to the box consists entirely of nodes in the centre and therefore `vanishes' or `retracts into the box'.
    In the second example the factor of $\delta$ arises due to a block that consists entirely of nodes in the centre and is not attached to the box.
\end{ex}

\begin{figure}[h!]
\centering

\[
\begin{tikzpicture}[x=1.5cm,y=-.5cm,baseline=-1.05cm]

\node[v] (a1) at (0,0) {};
\node[v] (a2) at (0,1) {};
\node[v] (a3) at (0,2) {};
\node[v] (a4) at (0,3) {};
\node[v] (a5) at (0,4) {};

\node[v] (b1) at (1,0) {};
\node[v] (b2) at (1,1) {};
\node[v] (b3) at (1,2) {};
\node[v] (b4) at (1,3) {};
\node[v] (b5) at (1,4) {};

\draw[e] (a1) to[out=0, in=180] (b3);
\draw[e] (a1) to[out=0, in=0] (a2);
\draw[e] (a2) to[out=0, in=0]   (a4);
\draw[e] (b3) to[out=180, in=180] (b4);
\draw[e] (a5) to[out=0, in=0] (a3);
\draw[e] (b1) to[out=180,in=180] (b2);

\end{tikzpicture}
\quad
\cdot
\quad
\begin{tikzpicture}[x=1.5cm,y=-.5cm,baseline=-1.05cm]

\node[v] (b1) at (0,0) {};
\node[v] (b2) at (0,1) {};
\node[v] (b3) at (0,2) {};
\node[v] (b4) at (0,3) {};
\node[v] (b5) at (0,4) {};

\node[v] (c1) at (1,0) {};
\node[v] (c2) at (1,1) {};
\node[v] (c3) at (1,2) {};
\node[v] (c4) at (1,3) {};
\node[v] (c5) at (1,4) {};

\draw[e] (b2) to[out=0,in=0] (b5);
\draw[e] (b1) to[out=0,in=180] (c1);
\draw[e] (b4) to[out=0, in=180] (c4);
\draw[e] (b3) to[out=0, in=180] (c2);
\draw[e] (c5) to[out=180,in=180] (c3);

\draw[e,fill=white] (0.9,-0.2) rectangle (1.1,2.2);
\node at (c2) {$3$};
\end{tikzpicture}
\quad
=
\quad
\begin{tikzpicture}[x=1.5cm,y=-.5cm,baseline=-1.05cm]

\node[v] (a1) at (0,0) {};
\node[v] (a2) at (0,1) {};
\node[v] (a3) at (0,2) {};
\node[v] (a4) at (0,3) {};
\node[v] (a5) at (0,4) {};

\node[v] (b1) at (1,0) {};
\node[v] (b2) at (1,1) {};
\node[v] (b3) at (1,2) {};
\node[v] (b4) at (1,3) {};
\node[v] (b5) at (1,4) {};

\node[v] (c1) at (2,0) {};
\node[v] (c2) at (2,1) {};
\node[v] (c3) at (2,2) {};
\node[v] (c4) at (2,3) {};
\node[v] (c5) at (2,4) {};

\draw[e] (a1) to[out=0, in=180] (b3);
\draw[e] (a1) to[out=0, in=0] (a2);
\draw[e] (a2) to[out=0, in=0]   (a4);
\draw[e] (b3) to[out=180, in=180] (b4);
\draw[e] (a5) to[out=0, in=0] (a3);
\draw[e] (b1) to[out=180,in=180] (b2);

\draw[e] (b2) to[out=0,in=0] (b5);

\draw[e] (b4) to[out=0, in=180] (c4);
\draw[e] (b1) to[out=0,in=180] (c1);
\draw[e] (c5) to[out=180,in=180] (c3);
\draw[e] (b3) to[out=0, in=180] (c2);

\draw[e,fill=white] (1.9,-0.2) rectangle (2.1,2.2);
\node at (c2) {$3$};

\end{tikzpicture}
\quad
=  
\quad
\begin{tikzpicture}[x=1.5cm,y=-.5cm,baseline=-1.05cm]

\node[v] (a1) at (0,0) {};
\node[v] (a2) at (0,1) {};
\node[v] (a3) at (0,2) {};
\node[v] (a4) at (0,3) {};
\node[v] (a5) at (0,4) {};

\node[v] (b1) at (1,0) {};
\node[v] (b2) at (1,1) {};
\node[v] (b3) at (1,2) {};
\node[v] (b4) at (1,3) {};
\node[v] (b5) at (1,4) {};

\draw[e] (a1) to[out=0, in=180] (b4);
\draw[e] (a1) to[out=0,in=180] (b2);
\draw[e] (a1) to[out=0, in=0] (a2);
\draw[e] (a2) to[out=0, in=0]   (a4);
\draw[e] (a5) to[out=0, in=0] (a3);
\draw[e] (b5) to[out=180,in=180] (b3);

\draw[e,fill=white] (0.9,-0.2) rectangle (1.1,2.2);
\node at (b2) {$3$};

\end{tikzpicture}
=0
\]
\vspace{1cm}
\[
\begin{tikzpicture}[x=1.5cm,y=-.5cm,baseline=-1.05cm]

\node[v] (a1) at (0,0) {};
\node[v] (a2) at (0,1) {};
\node[v] (a3) at (0,2) {};
\node[v] (a4) at (0,3) {};
\node[v] (a5) at (0,4) {};

\node[v] (b1) at (1,0) {};
\node[v] (b2) at (1,1) {};
\node[v] (b3) at (1,2) {};
\node[v] (b4) at (1,3) {};
\node[v] (b5) at (1,4) {};

\draw[e] (a1) to[out=0, in=180] (b3);
\draw[e] (a1) to[out=0, in=0] (a2);
\draw[e] (a2) to[out=0, in=0]   (a3);
\draw[e] (b3) to[out=180, in=180] (b4);
\draw[e] (a5) to[out=0, in=0] (a4);
\draw[e] (b2) to[out=180,in=180] (b5);
\draw[e] (a4) to[out=0, in=180] (b1);

\end{tikzpicture}
\quad
\cdot
\quad
\begin{tikzpicture}[x=1.5cm,y=-.5cm,baseline=-1.05cm]

\node[v] (b1) at (0,0) {};
\node[v] (b2) at (0,1) {};
\node[v] (b3) at (0,2) {};
\node[v] (b4) at (0,3) {};
\node[v] (b5) at (0,4) {};

\node[v] (c1) at (1,0) {};
\node[v] (c2) at (1,1) {};
\node[v] (c3) at (1,2) {};
\node[v] (c4) at (1,3) {};
\node[v] (c5) at (1,4) {};

\draw[e] (b2) to[out=0,in=0] (b5);
\draw[e] (b1) to[out=0,in=180] (c1);
\draw[e] (b4) to[out=0, in=180] (c4);
\draw[e] (b3) to[out=0, in=180] (c2);
\draw[e] (c5) to[out=180,in=180] (c3);

\draw[e,fill=white] (0.9,-0.2) rectangle (1.1,2.2);
\node at (c2) {$3$};
\end{tikzpicture}
\quad
=
\quad
\begin{tikzpicture}[x=1.5cm,y=-.5cm,baseline=-1.05cm]

\node[v] (a1) at (0,0) {};
\node[v] (a2) at (0,1) {};
\node[v] (a3) at (0,2) {};
\node[v] (a4) at (0,3) {};
\node[v] (a5) at (0,4) {};

\node[v] (b1) at (1,0) {};
\node[v] (b2) at (1,1) {};
\node[v] (b3) at (1,2) {};
\node[v] (b4) at (1,3) {};
\node[v] (b5) at (1,4) {};

\node[v] (c1) at (2,0) {};
\node[v] (c2) at (2,1) {};
\node[v] (c3) at (2,2) {};
\node[v] (c4) at (2,3) {};
\node[v] (c5) at (2,4) {};

\draw[e] (a1) to[out=0, in=180] (b3);
\draw[e] (a1) to[out=0, in=0] (a2);
\draw[e] (a2) to[out=0, in=0]   (a3);
\draw[e] (b3) to[out=180, in=180] (b4);
\draw[e] (a5) to[out=0, in=0] (a4);
\draw[e] (b2) to[out=180,in=180] (b5);
\draw[e] (a4) to[out=0, in=180] (b1);

\draw[e] (b2) to[out=0,in=0] (b5);

\draw[e] (b4) to[out=0, in=180] (c4);
\draw[e] (b1) to[out=0,in=180] (c1);
\draw[e] (c5) to[out=180,in=180] (c3);
\draw[e] (b3) to[out=0, in=180] (c2);

\draw[e,fill=white] (1.9,-0.2) rectangle (2.1,2.2);
\node at (c2) {$3$};

\end{tikzpicture}
\quad
=  
\quad
\delta\cdot
\begin{tikzpicture}[x=1.5cm,y=-.5cm,baseline=-1.05cm]

\node[v] (a1) at (0,0) {};
\node[v] (a2) at (0,1) {};
\node[v] (a3) at (0,2) {};
\node[v] (a4) at (0,3) {};
\node[v] (a5) at (0,4) {};

\node[v] (b1) at (1,0) {};
\node[v] (b2) at (1,1) {};
\node[v] (b3) at (1,2) {};
\node[v] (b4) at (1,3) {};
\node[v] (b5) at (1,4) {};

\draw[e] (a1) to[out=0, in=180] (b4);
\draw[e] (a1) to[out=0,in=180] (b2);
\draw[e] (a1) to[out=0, in=0] (a2);
\draw[e] (a2) to[out=0, in=0]   (a3);
\draw[e] (a5) to[out=0, in=0] (a4);
\draw[e] (b5) to[out=180,in=180] (b3);
\draw[e] (a4) to[out=0,in=180] (b1);

\draw[e,fill=white] (0.9,-0.2) rectangle (1.1,2.2);
\node at (b2) {$3$};

\end{tikzpicture}
\]
\caption{The module structure of $\P_5\otimes_{\P_3}\t$}
\label{fig:MultInd2}
\end{figure}
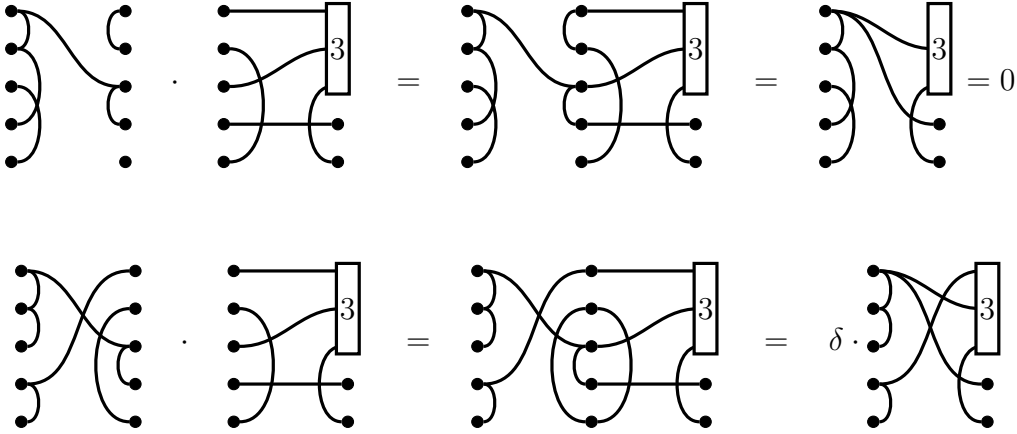

\section{The principle of inductive resolutions}\label{sec-inductive-principle}

In this brief section we state an abstract form of the principle that underlies the technique of inductive resolutions that appeared in~\cite{BoydHepworthStability} and~\cite{BHP2021}.
It allows us to identify modules that vanish under a fixed functor of the form $\Tor_i^A(M,-)$ by resolving them using modules that already have this property, hence the name `inductive resolutions'.
The theorem below is an abstraction of Section 3.3 of~\cite{BoydHepworthStability}.
It can be regarded as an application of the general principle that a derived functor can be computed using resolutions by objects that are acyclic for that derived functor.

\begin{thm}\label{inductive-resolution-method}
    Let $A$ be an algebra over a ring $R$, and let $M$ be a right $A$-module.
    Suppose that $N$ is a left $A$-module equipped with a resolution $Q_\ast\to N$ with the following two properties:
    \begin{itemize}
        \item
        $\Tor^A_\ast(M,Q_j)$ vanishes in positive degrees for all $j\geqslant 0$.
        \item
        $M\otimes_A Q_\ast\to M\otimes_A N$ is a resolution. 
    \end{itemize}
    Then $\Tor^A_\ast(M,N)$ vanishes in positive degrees.
\end{thm}

\begin{proof}
    Let $P_\ast\to M$ be a projective resolution, so that for any left $A$-module $B$, the groups $\Tor_\ast^A(M,B)$ are computed by the complex $P_\ast\otimes_A B$.
    Consider the double complex $P_\ast\otimes_A Q_\ast$.
    There are two natural spectral sequences converging to the homology of the totalisation $\Tot(P_\ast\otimes_A Q_\ast)$.
    For more on these spectral sequences, see Section~5.6 of~\cite{Weibel} and the summary in Section~3.2 of~\cite{BoydHepworthStability}.
    
    The first spectral sequence has $E^1$-term given by 
    $$
        \IE^1_{i,j} 
        = H_j(P_i\otimes_A Q_\ast)
        \cong P_i\otimes_A H_j(Q_\ast)
    $$
    with~$d^1$ induced by the differential of $P$. 
    The isomorphism holds because each~$P_i$ is projective and therefore flat.
    It follows that the~$E^2$-term is
    \[
        \IE^2_{i,j} = \Tor_i^A(M,H_j(Q_\ast)).
    \]
    Since $Q_\ast$ is a resolution of $N$, it follows that $\IE^2_{\ast,\ast}$ is simply $\Tor_\ast^A(M,N)$ concentrated on the $x$-axis, so that the same is true of $\IE^\infty_{\ast,\ast}$, and therefore we conclude that $H_\ast(\Tot(P_\ast\otimes_A Q_\ast))\cong\Tor_\ast^A(M,N)$.
    
    The second spectral sequence has~$E^1$-term given by~$\IIE^1_{i,j} = H_j(P_\ast\otimes_AQ_i)$, \emph{i.e.}
    \[
        \IIE^1_{i,j} =  \Tor^A_j(M,Q_i)
    \]
    with~$d^1$ induced by the boundary maps of~$Q_\ast$.
    Our first assumption now shows that $\IIE^1_{\ast,\ast}$ is concentrated on the $x$-axis, where it is given by $\Tor^A_0(M,Q_\ast) = M\otimes_A Q_\ast$.
    Consequently $\IIE^2_{\ast,\ast}$ is given by the homology of $M\otimes_A Q_\ast$, which by our second assumption is simply a copy of $M\otimes_A N$ at the origin.
    This shows that the homology of $\Tot(P_\ast\otimes_A Q_\ast)$ is simply a copy of $M\otimes_A N$ in degree $0$.

    Comparing the outcomes of the two spectral sequences, we see that $\Tor_\ast^A(M,N)$ vanishes in positive degrees, as required. 
\end{proof}

\section{Inductive resolutions}\label{sec-inductive resolutions}

In this section, we will use the technique of \emph{inductive resolutions}, which originated in~\cite{BoydHepworthStability} and was further used in~\cite{BHP2021}.

\begin{Def}\label{Def:Jx}
    Suppose that $X$ is a subset of the set $\{1,\dots,n\}$.
    Define $J_X$ to be the left-ideal in $\P_n$ that is the $R$-span of all diagrams in which, among the nodes on the right labelled by elements of $X$, there is at least one singleton or one pair of nodes that are in the same block. For $m\le n$, let $J_{\{n-m+1,\dots,n\}}$ be denoted by $J_m$.
\end{Def}

Observe that $J_n=J_{\{1,\dots,n\}}$ is the span of precisely the diagrams that are \emph{not} permutation diagrams. 
It is therefore the kernel of the projection map $\pi\colon\P_n\to R\fS_n$.

Our aim in this section is to prove the following theorem, which will be used in the final section to understand the $\Tor$ groups $\Tor^{\P_n}_\ast(\t,\P_n\otimes_{\P_m}\t)$.

\begin{thm}\label{theorem-inductive}
    Let $X\subseteq\{1,\dots,n\}$ and suppose that one of the following conditions holds:
    \begin{itemize}
        \item
        $|X|\leq n$ and $\delta$ is invertible in $R$.
        \item
        $|X|<n$.
    \end{itemize}
    Then the groups $\Tor_\ast^{\P_n}(\t,\P_n/J_X)$ vanish in positive degrees.
\end{thm}

The proof of \autoref{theorem-inductive} will occupy the rest of the section.
Aspects of the material are close to \cite[Section~3]{BoydHepworthStability} and \cite[Section~3]{BHP2021}, but overall the material here is significantly more complex.

Before we continue, we record the following lemma,
which extends \autoref{theorem-inductive} to degree $0$. 
We will need this lemma to prove the theorem.

\begin{lem}\label{tensor-P/J}
    Let $J$ be a left ideal of $\P_n$ that is included in $J_n$. Then
    \[ \t \otimes_{\P_n} \P_n/J \cong \t.\]
    In particular,
    \[ \Tor^{\P_n}_0(\t,\P_n/J_X) \cong \t\]
    for all $X\subset \{1, \dots, n\}$.
\end{lem}

\begin{proof}
    Due to the inclusions $0 \subset J \subset J_n$, we have the surjections
    \[ \t \otimes_{\P_n} \P_n \twoheadlongrightarrow \t \otimes_{\P_n} \P_n/J \twoheadlongrightarrow \t \otimes_{\P_n} \P_n/J_n.\]
    Because
    \[\t \otimes_{\P_n} \P_n \cong \t \]
    and 
    \[\t \otimes_{\P_n} \P_n/J_n \cong \t \otimes_{\P_n} R\fS_n \cong \t \otimes_{R\fS_n} R\fS_n \cong \t,\]
    the above composition is an isomorphism and the first map must be also injective.
\end{proof}

\subsection{Reducing to $\calA_{X,x}$ and $\calB_{X,x}$}

Our proof of \autoref{theorem-inductive} will be by induction on the cardinality of $X$. 
Ideally we would prove the inductive step by resolving $\P_n/J_X$ in terms of modules $\P_n/J_{X'}$ with $|X'|<|X|$.
However, we were not able to find a straightforward argument along these lines. 
To organise the argument, in this subsection we introduce some intermediate modules, and later on we will build our resolutions with these.

\begin{defn}
    Let $Y\subseteq X\subseteq\{1,\dots,n\}$, and let $x\in X$ and $y\in Y$.
    We define three left $\P_n$-submodules of $\P_n$ as follows:
    \begin{itemize}
        \item
        $A_x$ is the span of all diagrams in which $x$ is a singleton.
        \item
        $B_{X,x}$ is the span of all diagrams in which $x$ lies in the same block as some other element of $X$.
        \item
        $M_Y$ is the span of all diagrams in which the elements of $Y$ lie in the same block.
    \end{itemize}
    We define quotients of these as follows:
    \begin{gather*}
        \calA_{X,x}=\frac{A_x}{A_x\cap J_{X-\{x\}}},
        \qquad
        \calB_{X,x} = \frac{B_{X,x}}{B_{X,x}\cap J_{X-\{x\}}},
        \qquad
        \calM_{X,Y}=\frac{M_Y}{M_Y\cap J_{X-Y}}
    \end{gather*}
\end{defn}

The following result will be useful when we come to verify the second condition of \autoref{inductive-resolution-method}.

\begin{lem}\label{modules-lemma}
    The modules $\calA_{X,x}$, $\calB_{X,x}$, and $\calM_{X,Y}$ behave as follows under tensor product with $\t$.
    \begin{itemize}
        \item
        Let $x\in X\subseteq\{1,\dots,n\}$, and suppose that $n\geq 2$.
        Then $\t\otimes_{\P_n}\calA_{X,x}=0$.
        \item
        Let $x\in X\subseteq\{1,\dots,n\}$.
        Then $\t\otimes_{\P_n}\calB_{X,x}=0$.
        \item
        Let $Y\subseteq X\subseteq\{1,\dots,n\}$ with $|Y|\geq 2$.
        Then $\t\otimes_{\P_n}\calM_{X,Y}=0$ and $\calM_{X,Y}$ is a direct summand of $\P_n/J_{X-Y}$.
    \end{itemize}
\end{lem}

\begin{proof}
    We will show that, under the relevant conditions, each of $A_x$, $B_{X,x}$ and $M_Y$ vanishes under $\t\otimes_{\P_n}\!\!-$, and the same will then follow for $\calA_{X,x}$, $\calB_{X,x}$ and $\calM_{X,Y}$.
    
    To show that  $\t\otimes_{\P_n}A_x=0$,
    we take a diagram $\alpha$ in $A_x$, so that $x$ is a singleton in $\alpha$.
    Let $\beta$ denote a diagram obtained from $\alpha$ by placing $x$ into the same block as some other element on the right. 
    (This is possible by the assumption that $n\geq 2$.)
    Then $\alpha = \beta\cdot T_x$, and $\beta$ acts as $0$ on $\t$, so that
    \[
        1\otimes \alpha = 1\otimes \beta\cdot T_x = 1\cdot\beta\otimes T_x = 0\otimes T_x = 0,
    \]
    noting that $T_{x}\in A_{x}$. Since $A_x$ is the span of such diagrams $\alpha$, this completes the proof.
    
    The argument for the other two modules is similar.
    For $\t\otimes_{\P_n}B_{X,x}$, we take a diagram $\alpha\in B_{X,x}$, so that $x$ is in the same block as some other element $w\in X$, and we use the factorisation $\alpha = \alpha\cdot V_{wx}$, noting that $V_{wx}\in B_{X,x}$.
    
    For $M_Y$ we take a diagram $\alpha\in M_Y$, so that all elements of $Y$ lie in the same block, and factorise it as $\alpha = \alpha\cdot V_Y$ where $V_Y\in M_Y$ is the diagram with blocks $-Y\cup Y$ and $\{-p,p\}$ for $p\not\in Y$; the assumption $|Y|\geq 2$ ensures that $\alpha$ acts as $0$ on $\t$.

    For the final claim about $\calM_{X,Y}$, we use the fact that the element $V_Y$ above is idempotent and sends $J_{X-Y}$ into itself.
\end{proof}

The following proposition breaks down the problem of resolving $\P_n/J_X$ into the analogous problem for $\calA_{X,x}$ and $\calB_{X,x}$.

\begin{prop}\label{A-B-resolution}
    Let $X\subseteq\{1,\dots,n\}$, let $x\in X$, and assume $n\geq 2$.
    The following sequence, in which all maps are induced by either an inclusion or an identity map, is a resolution of $\P_n/J_X$.
    \[\xymatrix@R=5pt{
        \dots
        \ar[r]
        &
        0
        \ar[r]
        &
        \calA_{X,x}
        \oplus
        \calB_{X,x}
        \ar[r]
        &
        \P_n/J_{X-\{x\}} 
        \ar[r]
        &
        \P_n/J_X
        \\
        &
        2
        &
        1
        &
        0
        &
        -1
    }\]
    Moreover, applying $\t\otimes_{\P_n}\!\!-$ to the sequence gives a resolution of $\t\otimes_{\P_n} \P_n/J_X$.
\end{prop}

\begin{proof}
    The map $\P_n/J_{X-\{x\}} \to \P_n/J_X$ is induced by the identity map on $\P_n$ and is well defined since $J_{X-\{x\}} \subset J_{X}$. The map $\calA_{X,x} \to \P_n/J_{X-\{x\}}$ is induced by the inclusion $A_x \subset \P_n$ and is well defined since $(A_x\cap J_{X- \{x\}}) \subset J_{X - \{x\}}$, and a similar argument holds for the map $\calB_{X,x} \to\P_n/J_{X-\{x\}}$.

    Surjectivity of the right hand map is immediate, giving exactness in degree $-1$.
    
    To show exactness in degree $0$, observe that the ideals $J_{X-\{x\}}\subseteq J_X$ are both spanned by certain diagrams, so that the kernel of $\P_n/J_{X-\{x\}}\to \P_n/J_X$ is spanned by those diagrams that lie in $J_X$ but not $J_{X-\{x\}}$.
    For a diagram to lie in $J_X$, some element of $X$ must be a singleton, or two elements of $X$ must lie in the same block. 
    For it to not \emph{also} be an element of $J_{X-\{x\}}$, there must only be one singleton, namely $x$, or only one pair of elements lying in the same block, of which one must be $x$.
    The diagrams with $x$ a singleton are precisely the diagrams that span $A_x$, the diagrams in which $x$ lies in the same block as some other element of $X$ are precisely those that span $B_{X,x}$, and the proof follows.
    
    To show exactness in degree $1$, after unravelling the definitions of $\calA_{X,x}$ and $\calB_{X,x}$, it is sufficient to show that if we have $a\in A_x$ and $b\in B_{X,x}$ with $a+b\in J_{X-\{x\}}$, then $a,b\in J_{X-\{x\}}$.
    This follows quickly from fact that $A_x$ and $B_{X,x}$ have no basis elements in common.
    
    To prove the second claim, we will show that after applying $\t\otimes_{\P_n}\!\!-$, the resolution becomes 
    \[\xymatrix@R=5pt{
        \dots
        \ar[r]
        &
        0
        \ar[r]
        &
        0
        \ar[r]
        &
        \t
        \ar[r]^{\mathrm{Id}}
        &
        \t
        \\
        &
        2
        &
        1
        &
        0
        &
        -1
    }\]
    so that the claim follows directly.
    The identification of the final two terms and the map between them follows from \autoref{tensor-P/J}.
    The terms in degree $1$ 
    vanish by \autoref{modules-lemma}.
    (This is where we use the assumption $n\geq 2$.)
\end{proof}

The last result, together with \autoref{inductive-resolution-method}, shows that, in order to prove vanishing of higher $\Tor$'s for $\P_n/J_X$ by induction, we must first do the same for $\calA_{X,x}$ and $\calB_{X,x}$.
In the following two subsections we will construct resolutions of these.

\subsection{Resolving $\calA_{X,x}$}

We now attempt to resolve $\calA_{X,x}$.
It will turn out that this requires different methods depending on which assumption from~\autoref{theorem-inductive} we use: that $\delta$ is invertible, or that $|X|<n$. 
Under the first assumption we have the following.

\begin{prop}\label{A-summand}
    Suppose that $X\subseteq\{1,\dots,n\}$ and that $\delta$ is invertible in $R$.
    Then the module $\calA_{X,x}$ is a direct summand of $\P_n/J_{X-\{x\}}$.
\end{prop}

\begin{proof}
    The element $\delta^{-1}T_x$ is an idempotent, thanks to the computation $T_x^2 = \delta T_x$. 
    Right multiplication by $\delta^{-1}T_x$ sends $J_{X-\{x\}}$ into itself, and therefore induces an idempotent endomorphism of $\P_n/J_{X-\{x\}}$.
    The image of this endomorphism consists of all left multiples of $T_x$, and this is precisely $\calA_{X,x}=\frac{A_x}{A_x\cap J_{X-\{x\}}}$ as in the second paragraph of the proof of~\autoref{modules-lemma}.
\end{proof}

The above result shows that, if $\P_n/J_{X-\{x\}}$ has vanishing higher $\Tor$'s, then so does $\calA_{X,x}$.
When $\delta$ is not invertible, we need a more elaborate method using the following resolution.

\begin{defn}[The resolution $C(X,x,y)$]
    Suppose that $X\subset\{1,\dots,n\}$ with $|X|<n$, let $x\in X$, and let $y\in \{1,\dots,n\}-X$.
    We define a resolution $C(X,x,y)\to \calA_{X,x}$ as in~\autoref{C-figure}.
    \begin{figure}
        \[\xymatrix{
            \vdots
            \ar[d]_-{(1-T_xV_{xy})}&
            &
            \\
            \P_n/J_{X-\{x\}}
            \ar[d]_{ T_{x}V_{xy}}
            &
            2
            \\
            \P_n/J_{X-\{x\}}
            \ar[d]_{( 1-T_{x}V_{xy})}
            &
            1
            \\
            \P_n/J_{X-\{x\}}
            \ar[d]_{ T_{x}}
            &
            0
            \\
            \calA_{X,x}
            &
            -1
        } \]
        \caption{The resolution $C(X,x,y)\to\calA_{X,x}$ }
        \label{C-figure}
    \end{figure}
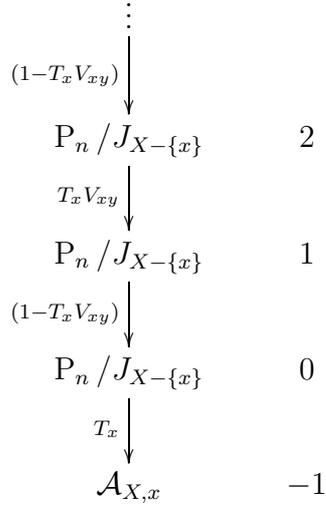
    Thus $C(X,x,y)$ is given by $\P_n/J_{X-\{x\}}$ in each degree. 
    The maps are all given by right-multiplication by the indicated elements, so that 
    the boundary maps alternate between $(1-T_xV_{xy})$ and $T_xV_{xy}$, and the augmentation $\P_n/J_{X-\{x\}}\to \calA_{X,x}$ is given by $T_x$.  
    The maps are well defined thanks to the fact that $T_xV_{xy}$ and $T_x$ send $J_{X-\{x\}}$ into itself; in the former case this follows from the fact that $y\not\in X$.

    To check that consecutive maps compose to $0$, one uses the computation $T_xV_{xy}T_x=T_x$ together with the resulting fact that $T_xV_{xy}$ is an idempotent.
    The fact that this really does define a resolution is given next. 
\end{defn}
    
\begin{prop}\label{C-resolution}
    Suppose that $X\subset\{1,\dots,n\}$ with $|X|<n$, let $x\in X$, and let $y\in \{1,\dots,n\}-X$.
    Assume that $n\geq 2$.
    Then
    \[
        C(X,x,y)\to\calA_{X,x}
        \qquad\text{and}\qquad
        \t\otimes_{\P_n}C(X,x,y)\to\t\otimes_{\P_n}\calA_{X,x}
    \]
    are both resolutions.
\end{prop}
    
\begin{proof}
    First, we must show that $C(X,x,y)\to\calA_{X,x}$ is acyclic.
    In degree $-1$ this is clear since $A_x$ consists of all left multiples of $T_x$.
    In degrees $1$ and above, this is an immediate consequence of the fact that $T_xV_{xy}$ is an idempotent.
    In degree $0$ we require a more complex argument, as follows.
    
    Suppose $\alpha$ is a diagram in which $x$ is a singleton.
    If $B$ is a block of $\alpha$ other than $\{x\}$, then we write $\alpha_B$ for the diagram obtained from $\alpha$ by incorporating $x$ into $B$.
    And we write $\alpha_y$ for the diagram $\alpha_{B_y}$, where $B_y$ is the block containing $y$.
    For example, the following diagrams show $\alpha$ with $B_y$ drawn in red and another block $B$ drawn in blue, together with $\alpha_B$ and $\alpha_y$.
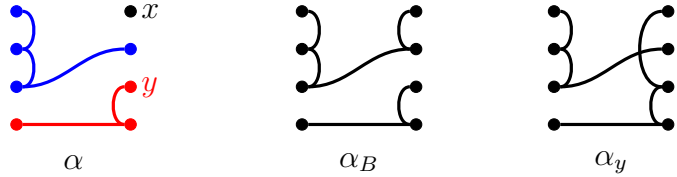
    \[
        \begin{tikzpicture}[x=1.5cm,y=-.5cm,baseline=-1.05cm]
            
            \node[v,blue] (a1) at (0,0) {};
            \node[v,blue] (a2) at (0,1) {};
            \node[v,blue] (a3) at (0,2) {};
            \node[v,red] (a4) at (0,3) {};
            
            \node[v] (b1) at (1,0) {};
            \node[v,blue] (b2) at (1,1) {};
            \node[v,red] (b3) at (1,2) {};
            \node[v,red] (b4) at (1,3) {};
            
            \draw[e,blue] (a3) to[out=0, in=180] (b2);
            \draw[e,blue] (a1) to[out=0, in=0] (a2);
            \draw[e,blue] (a2) to[out=0, in=0]   (a3);
    
            \draw[e,red] (a4) to[out=0,in=180] (b4);
            \draw[e,red] (b4) to[out=180,in=180] (b3);
    
            \node[right] () at (b1) {$x$};
            \node[right,red] () at (b3) {$y$};
            \node () at (0.5,4) {$\alpha$};
        \end{tikzpicture}
        \qquad\qquad
        \begin{tikzpicture}[x=1.5cm,y=-.5cm,baseline=-1.05cm]
            
            \node[v] (a1) at (0,0) {};
            \node[v] (a2) at (0,1) {};
            \node[v] (a3) at (0,2) {};
            \node[v] (a4) at (0,3) {};
            
            \node[v] (b1) at (1,0) {};
            \node[v] (b2) at (1,1) {};
            \node[v] (b3) at (1,2) {};
            \node[v] (b4) at (1,3) {};
            
            \draw[e] (a3) to[out=0, in=180] (b2);
            \draw[e] (a1) to[out=0, in=0] (a2);
            \draw[e] (a2) to[out=0, in=0]   (a3);
    
            \draw[e] (a4) to[out=0,in=180] (b4);
            \draw[e] (b4) to[out=180,in=180] (b3);
            \draw[e] (b1) to[out=180,in=180] (b2);
    
            \node () at (0.5,4) {$\alpha_B$};
        \end{tikzpicture}
        \qquad\qquad
        \begin{tikzpicture}[x=1.5cm,y=-.5cm,baseline=-1.05cm]
            
            \node[v] (a1) at (0,0) {};
            \node[v] (a2) at (0,1) {};
            \node[v] (a3) at (0,2) {};
            \node[v] (a4) at (0,3) {};
            
            \node[v] (b1) at (1,0) {};
            \node[v] (b2) at (1,1) {};
            \node[v] (b3) at (1,2) {};
            \node[v] (b4) at (1,3) {};
            
            \draw[e] (a3) to[out=0, in=180] (b2);
            \draw[e] (a1) to[out=0, in=0] (a2);
            \draw[e] (a2) to[out=0, in=0]   (a3);
    
            \draw[e] (a4) to[out=0,in=180] (b4);
            \draw[e] (b4) to[out=180,in=180] (b3);
            \draw[e] (b1) to[out=180,in=180] (b3);
    
            \node () at (0.5,4) {$\alpha_y$};
        \end{tikzpicture}
    \]
    Now observe that we have the relations
    \begin{equation} \label{alpha-to-alpha-y}
        \alpha - \delta \alpha_y = \alpha (1 - T_xV_{xy})
    \end{equation}
    and, for each block $B$ in $\alpha$,
    \begin{equation}\label{alpha-B-to-alpha-y}
        \alpha_B - \alpha_y = \alpha_B (1 - T_xV_{xy}).
    \end{equation}

    Now consider an element $a\in\P_n/J_{X-\{x\}}$ in degree $0$ that lies in the kernel of the augmentation $T_x\colon\P_n/J_{X-\{x\}}\to\calA_{X,x}$.
    We wish to show that $a$ is in the image of the differential, and we will do this by explaining how to adjust $a$ by elements in the image of the differential (which does not change the fact that it lies in the kernel) in order to reduce it to $0$.
    We can write $a$ as a linear combination of diagrams in $\P_n$ that do not lie in $J_{X-\{x\}}$, and these can be divided into the following cases:
    \begin{enumerate}
        \item
        Diagrams $\alpha$ in which $x$ is a singleton.
        Using elements of the form \eqref{alpha-to-alpha-y}, we may adjust $a$ by elements in the image of the differential in order to replace all such diagrams $\alpha$ with ones of the form $\alpha_y$.
        \item
        Diagrams in which $x$ is connected to some element outwith $X-\{x\}$.
        These diagrams all have the form $\alpha_B$, where $\alpha$ is the diagram obtained from the original by making $x$ a singleton, and $B$ is the block of $\alpha$ that originally contained $x$.
        Note that the assumption that the original diagram did not lie in $J_{X-\{x\}}$ means that $\alpha$ also does not lie in $J_{X-\{x\}}$.
        Using elements of the form \eqref{alpha-B-to-alpha-y}, we may adjust $a$ by elements in the image of the differential in order to replace all such diagrams $\alpha_B$ with ones of the form $\alpha_y$.
        \item
        Diagrams $\beta$ in which $x$ is connected to exactly one element, say $w$, in $X-\{x\}$.
        Then in $\beta T_x V_{xy}$ the element $w\in X-\{x\}$ is a singleton, so that $\beta T_x V_{xy}\in J_{X-\{x\}}$.
        Consequently $\beta = \beta(1-T_xV_{xy})$ in $\P_n/J_{X-\{x\}}$, and in particular $\beta$ lies in the image of the differential.
        We may therefore adjust $a$ by elements in the image of the differential to remove all diagrams of this form.
    \end{enumerate}
    After modifying $a$ as explained in each item above, we may now write it as a linear combination $a = \sum_\alpha \lambda_\alpha \alpha_y$ where $\alpha$ ranges over all diagrams that are not in $J_{X-\{x\}}$ and in which $x$ is a singleton.  
    We know that $a$ lies in the kernel of the differential, so that $a\cdot T_x=0$.
    However, we have 
    $a\cdot T_x = \sum_\alpha \lambda_\alpha \alpha_y\cdot T_x = \sum_\alpha \lambda_\alpha\alpha$, and since the $\alpha$ are distinct diagrams not in $J_{X-\{x\}}$, we can conclude that $\lambda_\alpha = 0$ for all $\alpha$, or in other words that $a=0$.
    This completes the argument in degree $0$, and so completes the proof that $C(X,x,y)\to \calA_{X,x}$ is a resolution.
    
    We now prove that $\t\otimes_{\P_n}C(X,x,y)\to\t\otimes_{\P_n}\calA_{X,x}$ is a resolution. 
    The target vanishes by \autoref{modules-lemma}, and $\t\otimes_{\P_n}\P_n/J_{X-\{x\}}=\t$ since $J_{X-\{x\}}$ acts as $0$ on $\t$.  
    Under the latter identification, the boundary maps, which used to be given by right-multiplication by the indicated elements, are now given by the action of those elements on $\t$, and therefore alternate between $0$ and $1$.
    The result follows.
\end{proof}

\subsection{Resolving $\calB_{X,x}$}

We now turn to the module $\calB_{X,x}$, for which we build the following resolution.

\begin{Def}[The resolution $D(X,x)\to\calB_{X,x}$]
    Let $X\subseteq\{1,\dots,n\}$ and let $x\in X$.
    Define an augmented complex $D(X,x)\to \calB_{X,x}$ as follows.
    In degree $i\geq 0$, $D(X,x)$ is given by 
    \[
        \bigoplus_{(x_0,\dots,x_i)}
        \calM_{X,\{x,x_0,\dots,x_i\}}
    \] 
    where the sum is over all tuples $(x_0,\dots,x_i)$ of \emph{distinct} elements of $X-\{x\}$.
    And on the summand corresponding to a tuple $(x_0,\ldots,x_i)$, the map $\delta_i$ is given by the map
    \[
        \calM_{X,\{x,x_0,\dots,x_i\}}
        \longrightarrow
        \calM_{X,\{x,x_0,\dots,x_{i-1}\}}
    \]
    obtained from the inclusions 
    \begin{align*}
        M_{\{x,x_0,\dots, x_i\}}&\hookrightarrow M_{\{x,x_0,\dots ,x_{i-1}\}}
        \\
        J_{X-\{x,x_0,\dots,x_i\}}&\hookrightarrow J_{X-\{x,x_0,\dots,x_{i-1}\}}.
    \end{align*}
    The map $\delta_i$ is simply the sum of these individual maps. To it put briefly, $\delta_i$ is the map that forgets that $x_i$ had to be in the same block as $x,x_0,\dots,x_{i-1}$.
    The complex is illustrated in \autoref{D-figure}.
    \begin{figure}
    \[
        \xymatrix{
            \vdots
            \ar[d]^-{\delta_3}
            \\
            \displaystyle\bigoplus_{(x_0,x_1,x_2)}
            \calM_{X,\{x,x_0,x_1,x_2\}}
            \ar[d]^-{\delta_2}
            &
            2
            \\
            \displaystyle\bigoplus_{(x_0,x_1)}
            \calM_{X,\{x,x_0,x_1\}}
            \ar[d]^-{\delta_1}
            &
            1
            \\
            \displaystyle\bigoplus_{(x_0)}
            \calM_{X,\{x,x_0\}}
            \ar[d]^-{\delta_0}
            &
            0
            \\
            \calB_{X,x}
            & -1
        }
    \]
    \caption{The resolution $D(X,x)\to \calB_{X,x}$.   Summations are over tuples of distinct elements of $X-\{x\}$.}
    \label{D-figure}
    \end{figure}
    Each composite $\delta_{i-1}\circ\delta_i$ vanishes, because any element in its image is a sum of diagrams that each contain two elements $x_{i-1},x_i\in X-\{x,x_1,\dots,x_{i-2}\}$ in the same block, and which therefore lie in $J_{X-\{x,x_1,\dots,x_{i-2}\}}$.
    We prove that this is indeed a resolution in \autoref{D-resolution}. 
\end{Def}

\begin{prop}\label{D-resolution}
    $D(X,x)\to \calB_{X,x}$ is indeed a resolution,
    and the same is true for $\t\otimes_{\P_n}D(X,x)\to \t\otimes_{\P_n}\calB_{X,x}$.
\end{prop}

\begin{proof}
    We first prove that $D(X,x) \to \calB_{X,x}$ is acyclic.
    
    In degree $-1$ we must show that $\delta_0$ is surjective. 
    A diagram in $B_{X,x}$ has $x$ in the same block as some other element $x_0$ of $X-\{x\}$, and therefore lies in the image of the inclusion $M_{\{x,x_0\}}\hookrightarrow B_{X,x}$, and surjectivity follows.
    
    In degree $0$, we first observe that if we consider any two summands in degree $0$, then their images under $\delta_0$ have trivial intersection.
    Indeed, this follows quickly from the fact that if $x_0$ and $x_0'$ are distinct elements of $X-\{x\}$, then $M_{\{x,x_0\}}\cap M_{\{x,x_0'\}}\subseteq J_{X-\{x\}}$, which itself holds because a diagram in  $M_{\{x,x_0\}}\cap M_{\{x,x_0'\}}$ has $x_0$ and $x_0'$ in the same block as $x$, and therefore in the same block as one another.
    So to prove exactness in degree $0$ we can look at just one $x_0$-summand at a time:
    \[
        \xymatrix{
            \displaystyle\bigoplus_{x_1\in X-\{x,x_0\}}
            \calM_{X,\{x,x_0,x_1\}}
            \ar[d]^-{\delta_1}
            &
            1
            \\
            \calM_{X,\{x,x_0\}}
            \ar[d]^-{\delta_0}
            &
            0
            \\
            \displaystyle
            \calB_{X,x}
            & -1
        }
    \]
    To prove that this sequence is exact at its middle term, observe that the kernel of $\delta_0$ is spanned by diagrams in $M_{\{x,x_0\}}$ that lie in $J_{X-\{x\}}-J_{X-\{x,x_0\}}$.
    Pick such a diagram.
    For the diagram to lie in $J_{X-\{x\}}$, two elements of $X-\{x\}$ must be in the same block, or an element of  $X-\{x\}$ must be a singleton.
    For the diagram to lie outwith $J_{X-\{x,x_0\}}$, since~$x_0$ cannot be a singleton in $M_{\{x,x_0\}}$, we conclude that $x_0$ must be in the same block as some other element of  $X-\{x\}$.
    So, $x_0$ lies in the same block as some element $x_1\in X-\{x,x_0\}$, and since the diagram is in $M_{\{x,x_0\}}$ it follows that $x,x_0,x_1$ must all be in the same block.
    Thus the diagram is in $M_{\{x,x_0,x_1\}}$, and so lies in the image of $\delta_1$.
    
    To prove exactness in degree $i\geqslant 1$ and above, one first observes that in degrees $i-1,i,i+1$ the complex splits as a direct sum over $(x_0,\ldots,x_{i-1})$.
    It is therefore enough to concentrate on a single $(x_0,\ldots,x_{i-1})$-summand at a time.
    Having restricted to such a summand, one now proves exactness similarly to the degree $0$ case, and we leave the details of this to the reader.

    The fact that the resolution remains acyclic after applying $\t\otimes_{\P_n}\!\!-$ follows immediately from \autoref{modules-lemma}, which shows that in fact the resolution vanishes under this operation.
\end{proof}

\subsection{Proof of  Theorem~\ref{theorem-inductive}}
    We first tackle the cases $n=0,1$.
    When $n=0$ we have $\P_n=R$ and the claim follows immediately.
    When $n=1$, we either have $X=\emptyset$, or we have $X=\{1\}$ and $\delta$ invertible.
    When $X=\emptyset$ we have $J_X=0$, so that $\P_n/J_X = \P_n$ and the claim follows.
    Finally, when $X=\{1\}$ and $\delta$ is invertible, then $J_X$ is the $R$-span, and indeed the $\P_n$-span, of the idempotent $\delta^{-1}T_{1}$. 
    Thus $J_X$ and $\P_n/J_X$ are both direct summands of $\P_n$, and in particular the latter is projective, so that the claim follows.
    
    We now assume that $n\geq 2$, and prove the claim by strong induction on the cardinality of $X$.
    When $X=\emptyset$ we have $J_X=0$, so that $\P_n/J_X = \P_n$ and the claim is immediate.
    
    Suppose now that $|X|>0$ and that the claim holds for all $X'$ of a smaller cardinality.
    According to \autoref{A-B-resolution} and \autoref{inductive-resolution-method}, it will be sufficient to show that the modules
    \[
        \P_n/J_{X-\{x\}} 
        \qquad
        \calA_{X,x}
        \qquad
        \calB_{X,x}
    \]
    all vanish under $\Tor^{\P_n}_i(\t,-)$ for $i>0$.
    
    In the case of $\P_n/J_{X-\{x\}}$ we have $\Tor^{\P_n}_i(\t,\P_n/J_{X-\{x\}})=0$ by the inductive hypothesis.
    
    For $\calA_{X,x}$,  we divide into the case where $\delta$ is invertible, and the case where $|X|<n$.  
    When $\delta$ is invertible, \autoref{A-summand} shows that $\calA_{X,x}$ is a direct summand of $\P_n/J_{X-\{x\}}$, which vanishes under $\Tor^{\P_n}_i(\t,-)$ by the inductive hypothesis, so that $\calA_{X,x}$ does as well.
    When $|X|<n$, \autoref{C-resolution} gives us resolutions 
    $C(X,x,y)\to\calA_{X,x}$
    and 
    $\t\otimes_{\P_n}C(X,x,y)\to\t\otimes_{\P_n}\calA_{X,x}$.
    The terms of $C(X,x,y)$ are all $\P_n/J_{X-\{x\}}$, which vanish under $\Tor^{\P_n}_i(\t,-)$ by the inductive hypothesis, so that \autoref{inductive-resolution-method} applies to tell us that the same is true for $\calA_{X,x}$ itself.
    
    For $\calB_{X,x}$, \autoref{D-resolution} gives us the resolutions
    $D(X,x)\to \calB_{X,x}$ 
    and 
    $\t\otimes_{\P_n}D(X,x)\to \t\otimes_{\P_n}\calB_{X,x}$.
    The terms of $D(X,x)$ are direct sums of modules of the form $\calM_{X,\{x,x_0,\dots,x_i\}}$. 
    Each $\calM_{X,\{x,x_0,\dots,x_i\}}$ is a direct summand of  $\P_n/J_{X-\{x,x_0,\dots,x_i\}}$ by \autoref{modules-lemma}, and since $\Tor_i^{\P_n}(\t,-)$ vanishes on the latter, it also vanishes on the former.
    (Note that this is the only place in our argument where we have used \emph{strong} induction.)
    We can now apply \autoref{inductive-resolution-method} to $D(X,x)$ to find that $\calB_{X,x}$ vanishes under $\Tor_i^{\P_n}(\t,-)$ as required.\qed

\section{Replacing Shapiro's Lemma}\label{sec-replacing shapiros lemma}
This section closely follows Section 4 of \cite{BHP2021}. We include all statements, and proofs of the lemmas which slightly differ in the case of partition algebras. The proof of \autoref{thm-tor-quotient-Homology-Sm} is identical to that in~\cite{BHP2021}, with adapted inputs.

As in the case for the Brauer algebras, we have inclusion and projection maps
\[
    R\fS_m\xrightarrow{\iota} \P_m\xrightarrow{\pi} R\fS_m.
\]
These are compatible with the inclusions $\P_{m}\to\P_n$ and $R\fS_{m}\to R\fS_n$, and also respect the actions on the trivial module.  
They therefore induce the following maps of $\Tor$-groups.
\[
    \Tor_\ast^{R\fS_n}(\t,R\fS_n\otimes_{R\fS_m}\t)
    \xrightarrow{\iota_\ast}
    \Tor_\ast^{\P_n}(\t,\P_n\otimes_{\P_m}\t)
    \xrightarrow{\pi_\ast}
    \Tor_\ast^{R\fS_n}(\t,R\fS_n\otimes_{R\fS_m}\t)
\]

Then the main result of this section is the following Theorem, which replaces Shapiro's Lemma in the Quillen style proof of homological stability for groups.

\begin{thm}\label{thm-tor-quotient-Homology-Sm}
Let $n\geqslant m\geqslant 0$. 
Suppose that $\delta$ is invertible in $R$, or that $m<n$.
Then the maps
\[
    \iota_\ast\colon
    \Tor_\ast^{R\fS_n}(\t,R\fS_n\otimes_{R\fS_m}\t)
    \longrightarrow
    \Tor_\ast^{\P_n}(\t,\P_n\otimes_{\P_m}\t)
\]
and
\[
    \pi_\ast\colon
    \Tor_\ast^{\P_n}(\t,\P_n\otimes_{\P_m}\t)
    \longrightarrow
    \Tor_\ast^{R\fS_n}(\t,R\fS_n\otimes_{R\fS_m}\t)
\]
are mutually inverse isomorphisms.
\end{thm}

\autoref{thma} follows immediately from \autoref{thm-tor-quotient-Homology-Sm} by taking $\delta$ invertible and $m=n$, using the identifications    $R\fS_n\otimes_{R\fS_m}\t\cong\t$ and  $\P_n\otimes_{\P_m}\t\cong \t$.

The remainder of this section is devoted to proving \autoref{thm-tor-quotient-Homology-Sm}, which follows in exactly the same way as Theorem 4.1 of \cite{BHP2021} after some preparatory definitions and lemmas.

Recall from \autoref{Def:Jx} that~$J_m\subseteq \P_n$ denotes the ideal consisting of all diagrams in which, among the nodes on the right labelled by $\{n-m+1,\dots, n\}$, there is a least one singleton or one pair of nodes in the same block. Observe that $\P_n$ is a right $R\fS_m$-module, via the inclusions $R\fS_m\subseteq \P_m\subseteq\P_n$, and that this module structure preserves $J_m$, since right multiplying by a diagram which permutes the nodes~$\{n-m+1,\dots, n\}$ does not change whether there exists a singleton or two nodes in the same block in this set. Therefore we have that $\P_n/J_m$ is a right $R\fS_m$-module.

\begin{lem}\label{lem-quotient-by-Jm-free}
For~$m\leq n$, $\P_n/J_m$ is free when regarded as a right $R\fS_m$-module.
\begin{proof}
$\P_n/J_m$ has basis consisting of the diagrams for which the nodes in $\{n-m+1,\dots, n\}$ have no singleton, and no two nodes in the same block. This means that each node in $\{n-m+1,\dots, n\}$ is attached to a distinct block in the diagram.
Now, $\fS_m$ acts freely on this basis, since multiplying such a diagram with a permutation in~$\fS_m$ results again in a diagram where the nodes in $\{n-m+1,\dots, n\}$ are attached to distinct blocks. Under this action, the stabilizer of any such diagram is trivial.
\end{proof}
\end{lem}

\begin{lem}\label{lem-quotient-by-Jm-tensor-product}
    For~$m\leq n$, there is an isomorphism of left $\P_n$-modules
    \[
    \P_n/J_m\otimes_{R\fS_m}\t \cong \P_n\otimes_{\P_m}\t 
    \]
    under which $(b+J_m)\otimes r\in\P_n/J_m\otimes_{R\fS_m}\t$ corresponds to $b\otimes r\in \P_n\otimes_{\P_m}\t$.
\end{lem}

\begin{proof}
Throughout this proof we regard~$J_m$ as an ideal in~$\P_n$, and write~$J_m\cap \P_m$ for the corresponding ideal in~$\P_m$.

Let us show that the maps
\[ (b+J_m)\otimes r \longmapsto b\otimes r\]
and
\[ b\otimes r \longmapsto (b+J_m)\otimes r\]
are well-defined. It then immediately follows that they are inverses and thus isomorphisms.

For the first map, we need to show that $b\sigma \otimes r = b\otimes r$ for $\sigma \in \fS_m$ and that $j \otimes r \in \P_n \otimes_{\P_m} \t$ is zero if $j \in J_m$. The first equation follows immediately as $\sigma \in \fS_m \subset \P_m$ acts as the identity on $\t$. 
The second condition holds because if $j\in J_m$, then we can write $j$ as a sum of products of the form $b\cdot j'$ where $b\in\P_n$ and $j'\in J_m\cap\P_m$, and for each such summand we have $b\cdot j'\otimes r = b\otimes j'\cdot b = b\otimes 0 =0$. 

For the second map, we let $b\in \P_n$, $b'\in \P_m$, and $r\in \t$, and show that $(bb'+J_m) \otimes r = (b+J_m) \otimes (b'\cdot r)$. It is enough to prove this for $b'\in R\fS_m$ and $b'\in J_m\cap\P_m$ as $\P_m =  R\fS_m\oplus (J_m\cap\P_m)$. For $b'\in R\fS_m$, we get the equation directly from the definition of the tensor product. For $b'\in J_m\cap\P_m$, we note that $bb' \in J_m = P_n \cdot J_m$ and thus $(bb'+J_m) \otimes r$ is zero, as is $(b+J_m) \otimes (b'\cdot r)$ since $b'\cdot r=0$.
\end{proof}

Now recall from \autoref{theorem-inductive} that, under  the hypotheses of \autoref{thm-tor-quotient-Homology-Sm},
\[
\Tor^{\P_n}_\ast (\t, \P_n/J_m)=\begin{cases} \t &\mbox{if } \ast= 0 \\
0 & \mbox{if } \ast>0 \end{cases}.
\]

\begin{proof}[Proof of \autoref{thm-tor-quotient-Homology-Sm}]
The proof of \autoref{thm-tor-quotient-Homology-Sm} now follows exactly as in \cite[Proof of Theorem 4.1]{BHP2021}, replacing the occurrences of~$\Br_n$ with~$\P_n$, and inputting \autoref{lem-quotient-by-Jm-tensor-product} and \autoref{theorem-inductive} as appropriate.
\end{proof}

\section{High connectivity}\label{sec-high connectivity}

We build a complex similar to the one in \cite{HepworthIH} and \cite{BHP2021}. 

\begin{defn}\label{def: complex for Pn}
For~$n$ a non-negative integer, we define the chain complex $C_n=(C_n)_\ast$ of~$\P_n$-modules as follows. The degree~$p$ part $(C_n)_p$ is non-zero in degrees $-1\leq p \leq n-1$, where it is given by
\[
(C_n)_p=\P_n\otimes_{\P_{n-(p+1)}}\t.
\]
So in degree~$-1$ it follows that~$(C_n)_{-1}=\P_n\otimes_{\P_{n}}\t\cong\t$.
For~$0\leq p \leq n-1$ the degree~$p$ differential~$\partial^p$ is given by the alternating sum
\[ \partial^p = \sum_{i=0}^{p} (-1)^i d^p_i \colon (C_n)_p \longrightarrow (C_n)_{p-1}.\]
Where, algebraically, the map~$d^p_i$ for~${1}\leq i\leq p$ is given by
\begin{eqnarray*}
d^p_i \colon \P_n \tens[\P_{n-(p+1)}] \t &\longrightarrow& \P_n \tens[\P_{n-p}] \t\\
x\otimes r &\mapsto& (x\cdot S_{n-p+i-1}\dots S_{n-p})\otimes r
\end{eqnarray*}
and
\begin{eqnarray*}
d^p_0 \colon \P_n \tens[\P_{n-(p+1)}] \t &\longrightarrow& \P_n \tens[\P_{n-p}] \t\\
x\otimes r &\longmapsto& x\otimes r.
\end{eqnarray*}
In other words, when $i=0$ the product $S_{n-p+i-1}\dots S_{n-p}$ is taken to be the empty product, \emph{i.e.}~the identity element.

In terms of diagrams, elements in degree~$p$ can be described as diagrams with an $(n-(p+1))$-box at the top right, as in~\autoref{prop:HomBr} and the paragraph which follows it.
If we label the nodes below the $(n-(p+1))$-box by $0,\dots,p$ from top to bottom, then $d_i^p$ lifts up node $i$ and plugs it into the box.
\end{defn}

We now filter~$C_n$. Note that in~\cite{BHP2021} we first decomposed~$C_n$ based on the number of disjoint blocks on the left, and we could also do that here. However this is not necessary for the proof.

\begin{Def}\label{def-filtration}
We define a filtration 
\[
    F_0C_n
    \subseteq
    F_1C_n
    \subseteq
    \dots
    \subseteq
    F_{\lfloor\frac{n}{2}\rfloor}C_n = C_n
\]
of $C_n$ as follows. 
The~$j$th level $F_jC_n$ is generated by diagrams with at most $j$ blocks that have at least $2$ positive (right hand) nodes and are not connected to the box. Note that this is indeed a filtration, since the boundary map can only decrease the number of blocks on the right {not} connected to the box.
\end{Def}

We briefly recall the definition of the complex of injective words with separators.

\begin{Def}[Injective words with separators]
    Let $X$ be a finite set and let $k\geqslant 0$.
    An \emph{injective word on $X$ with $k$ separators} is a word with letters taken from the set $X\sqcup\{|\}$ consisting of $X$ and the \emph{separator} $|$, where each letter from $X$ appears at most once, and where the separator appears exactly $k$ times.
    When $k=0$, then these are simply the injective words on $X$.
\end{Def}

\begin{Def}[The complex of injective words with separators]
    Let $X$ be a finite set, let $s\geqslant 0$, and let $R$ be a commutative ring.
    The \emph{complex of injective words with $s$ separators} is the $R$-chain complex $W_X^{(s)}$ concentrated in degrees $-1\leq p \leq |X|-1$, and defined as follows.
    In degree $p$, $(W^{(s)}_X)_p$ has basis given by the injective words on $X$ with $s$ separators, and with $(p+1)$ letters from $X$. 
    Thus such a word ${\bf a} \in (W^{(s)}_X)_p$ has length~$s+p+1$. Let~$r=s+p$ and~${\bf a}=a_0a_1\dots a_r$.
    The boundary operator $\partial^p\colon (W^{(s)}_X)_p\to (W^{(s)}_X)_{p-1}$ is defined by the rule
    \[
        \partial^p(a_0a_1\dots a_r)=\sum_{i=0}^r(-1)^ia_0\dots\widehat{a_i}\dots a_r
    \]
    subject to the condition that if the omitted letter is a separator, then the corresponding term is omitted (or identified with $0$).
    In other words, the boundary is the signed sum of the words obtained by deleting the letters that come from $X$ and \emph{not deleting any separators}, but with signs determined by the position of the deleted letter among all letters \emph{including the separator}.:
    \[
        \partial^p(a_0a_1\dots a_r)=\sum_{a_i\in X}(-1)^ia_0\dots\widehat{a_i}\dots a_r
    \]
\end{Def}

We will aim to identify the filtration quotients~$F_jC_n/F_{j-1}C_n$ with a sum of shifted copies of the complex of injective words with separators, as in~\cite{BHP2021}. (Note that in \cite{BHP2021} the argument for the Brauer algebras is somewhat simpler, and so the reader may wish to look at the Brauer proof first.)

To complete this identification, we exhibit a one-to-one correspondence between diagrams and tuples of data. This correspondence is complicated, so we start with the simple example of the tuple corresponding to a diagram with no box, and no restriction on the right hand side blocks. Recall that a diagram is a pictorial way of representing a partition of the set~$\{-n, \ldots , -1,1, \ldots , n \}$.

A diagram $D$ determines, and is determined by,
    a tuple $(L,R,\phi)$ consisting of:
    \begin{itemize}
        \item
        A partition $L$ of $\{-1,\ldots,-n\}$.
        \item
        A partition $R$ of $\{1,\ldots,n\}$.
        \item
        A labelling $\phi\colon R\to \{\emptyset\}\cup L$ 
        with the property that $\phi(r)=\phi(r')$ 
        only when $r=r'$ or $\phi(r)=\phi(r')=\emptyset$.
    \end{itemize}
    The correspondence sends a diagram $D$ to the 
    tuple $(L,R,\phi)$ for which:
    \begin{itemize}
        \item
        $L$ is the induced partition on the left-hand
        nodes $-1,\ldots,-n$.
        \item
        $R$ is the induced partition on the right-hand
        nodes $1,\ldots,n$.
        \item
        $\phi$ labels a block on the right by the 
        (necessarily unique) block on the left to which
        it is attached, if any, and labels it by
        $\emptyset$ otherwise.
    \end{itemize}

    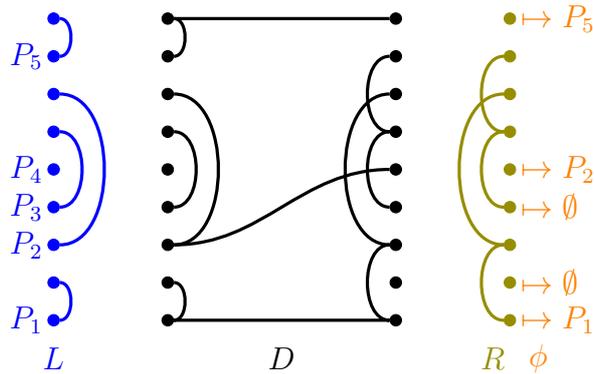
\begin{figure}[h!]
\centering
\[
    \begin{tikzpicture}[x=1.5cm,y=-.5cm,baseline=-2.05cm]
    
        \node[v] (a1) at (0,0) {};
        \node[v] (a2) at (0,1) {};
        \node[v] (a3) at (0,2) {};
        \node[v] (a4) at (0,3) {};
        \node[v] (a5) at (0,4) {};
        \node[v] (a6) at (0,5) {};
        \node[v] (a7) at (0,6) {};
        \node[v] (a8) at (0,7) {};
        \node[v] (a9) at (0,8) {};

        \node[v,blue] (c1) at (-1,0) {};
        \node[v,blue] (c2) at (-1,1) {}; 
        \node[v,blue] (c3) at (-1,2) {};
        \node[v,blue] (c4) at (-1,3) {}; 
        \node[v,blue] (c5) at (-1,4) {};
        \node[v,blue] (c6) at (-1,5) {}; 
        \node[v,blue] (c7) at (-1,6) {};
        \node[v,blue] (c8) at (-1,7) {}; 
        \node[v,blue] (c9) at (-1,8) {};

        \node[v] (b1) at (2,0) {};
        \node[v] (b2) at (2,1) {};
        \node[v] (b3) at (2,2) {};
        \node[v] (b4) at (2,3) {};
        \node[v] (b5) at (2,4) {};
        \node[v] (b6) at (2,5) {};
        \node[v] (b7) at (2,6) {};
        \node[v] (b8) at (2,7) {};
        \node[v] (b9) at (2,8) {};

        \node[v,olive] (e1) at (3,0) {};
        \node[v,olive] (e2) at (3,1) {};
        \node[v,olive] (e3) at (3,2) {};
        \node[v,olive] (e4) at (3,3) {};
        \node[v,olive] (e5) at (3,4) {};
        \node[v,olive] (e6) at (3,5) {};
        \node[v,olive] (e7) at (3,6) {};
        \node[v,olive] (e8) at (3,7) {};
        \node[v,olive] (e9) at (3,8) {};
        
        \draw[e] (a1) to[out=0, in=180] (b1);      
        \draw[e] (a7) to[out=0, in=180] (b5);
        \draw[e] (a1) to[out=0, in=0] (a2);
        \draw[e] (a8) to[out=0, in=0] (a9) to[out=0, in=180] (b9) to[out=180, in=180] (b7) to[out=180, in=180] (b3);
        \draw[e] (a3) to[out=0, in=0]   (a7);
        \draw[e] (a4) to[out=0, in=0]   (a6);
        \draw[e] (b2) to[out=180, in=180](b4) to[out=180, in=180] (b6);

        \draw[e,blue] (c1) to[out=0, in=0] (c2);
        \node[left,blue] at (c9) {$P_1$};
        \node[left,blue] at (c7) {$P_2$};
        \node[left,blue] at (c6) {$P_3$};
        \node[left,blue] at (c5) {$P_4$};
        \node[left,blue] at (c2) {$P_5$};
        \draw[e,blue] (c8) to[out=0, in=0] (c9);
        \draw[e,blue] (c3) to[out=0, in=0]   (c7);
        \draw[e,blue] (c4) to[out=0, in=0]   (c6);
        
        \draw[e,olive] (e9) to[out=180, in=180] (e7) to[out=180, in=180] (e3);
        \draw[e,olive] (e2) to[out=180, in=180] (e4) to[out=180, in=180] (e6);
        \node[right,orange] at (e1) {$\mapsto P_5$};
        \node[right,orange] at (e5) {$\mapsto P_2$};
        \node[right,orange] at (e6) {$\mapsto\emptyset$};
        \node[right,orange] at (e9) {$\mapsto P_1$};
        \node[right,orange] at (e8) {$\mapsto\emptyset$};

        \node[blue] at (-1,9) {$L$};
        \node at (1,9) {$D$};
        \node[olive] at (2.85,9) {$R$};
        \node[orange] at (3.25,9) {$\phi$};
    \end{tikzpicture}
\]
\caption{This figure shows the process of extracting from a diagram~$D$ the tuple $(L,R,\phi)$. Blocks in~$L$, and the labelings of~$R$ are indicated at their lowermost node.}
\label{fig:simplediagramtuple}
\end{figure}

    An example is shown in \autoref{fig:simplediagramtuple}. Here, the process of restricting the partition to the left and right sides of the diagram amounts to discarding
    all the left-to-right connections.
    Those left-to-right connections are instead recorded
    in the labelling $\phi$.
    To see that $\phi$ satisfies the third
    property above because, observe that if it did not, then the diagram would have two distinct blocks $r,r'$ on the right attached to the \emph{same} block on the left.
    That would be a contradiction because then $r$ and $r'$ 
    would in fact themselves be the same.

We now observe that the filtration quotient $F_jC_n/F_{j-1}C_n$ has a basis in degree~$p$ consisting of diagrams which have an~$(n-(p+1))$-box on the right, and exactly~$j$ blocks with $\geq 2$ nodes on the right that are not connected to the box. 
Here, the size of the box is determined by the degree
as in~\autoref{def: complex for Pn},
and the condition on the $j$ blocks follows from the definition
of the filtration given in~\autoref{def-filtration}.
An example is given in~\autoref{fig:example blob diagram}.
\begin{figure}[h!]
\centering
\[
    \begin{tikzpicture}[x=1.5cm,y=-.5cm,baseline=-2.05cm]
    
        \node[v] (a1) at (0,0) {};
        \node[v] (a2) at (0,1) {};
        \node[v] (a3) at (0,2) {};
        \node[v] (a4) at (0,3) {};
        \node[v] (a5) at (0,4) {};
        \node[v] (a6) at (0,5) {};
        \node[v] (a7) at (0,6) {};
        \node[v] (a8) at (0,7) {};
        \node[v] (a9) at (0,8) {};
        
        \node[v] (b3) at (2,2) {};
        \node[v] (b4) at (2,3) {};
        \node[v] (b5) at (2,4) {};
        \node[v] (b6) at (2,5) {};
        \node[v] (b7) at (2,6) {};
        \node[v] (b8) at (2,7) {};
        \node[v] (b9) at (2,8) {};

        \draw[e] (a1) to[out=0, in=180] (2,0);      
        \draw[e] (a7) to[out=0, in=180] (b5);
        \draw[e] (a1) to[out=0, in=0] (a2);
        \draw[e] (a8) to[out=0, in=0] (a9) to[out=0, in=180] (b9) to[out=180, in=180] (b7) to[out=180, in=180] (b3);
        \draw[e] (a3) to[out=0, in=0]   (a7);
        \draw[e] (a4) to[out=0, in=0]   (a6);
        \draw[e] (2,0.5) to[out=180, in=180](b4) to[out=180, in=180] (b6);
        \draw[fill=white, line width=1] (1.9,-0.2) rectangle (2.1,1.2);
        \node at (2,0.5) {$2$};
        \node at (1,9) {$D$};
    \end{tikzpicture}
\]
\caption{An example of a diagram~$D$, when $n=9$,~$j=1$ and $p=6$.}
\label{fig:example blob diagram}
\end{figure}
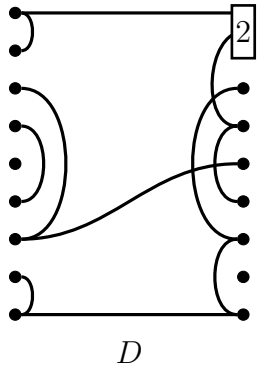
In the next definition, we explain how these basis diagrams determine a tuple of data, analogously to the discussion above. Once this data has been stripped from the diagram, we are left with the desired information of an injective word with separators. In this injective word, the letters encode left-to-right connections for which the block on the right has a single element; and the separators correspond to all other nodes below the box on the right.
There are at least $2j$ of these separators, because 
there are precisely $j$ blocks on the right that have $2$ or more right-hand nodes and are not connected to the box.
Later, in \autoref{lem:phi is an iso} we show how to conversely start with an injective word with separators and our tuple of data and rebuild the diagram.

\begin{Def}\label{def:5-tuple}
A diagram in the basis of $(F_jC_n/F_{j-1}C_n)_p$ determines a tuple
\[
    (P, X, s, Y, f)
\]
consisting of the following data:
\begin{itemize}
    \item
    A partition $P$ of~$\{-1,\dots ,-n\}$.
    \item    
    A subset $X$ of the blocks of $P$.
    \item
    A number~$2j\leq s \leq n -|X|$.
    \item 
    A partition $Y$ of $\{1,\dots, s\}$, such that~$\geq j$ blocks have size~$\geq 2$.
    \item 
    A labelling $f\colon Y \to (\{\emptyset\}\cup P\setminus X) \times \{\square, \neg\square\} $ (where the symbols $\square, \neg\square$ represent ``box'' and ``not-box'' respectively) satisfying the following properties:
    \begin{itemize}
        \item singletons have first label $\emptyset$
        \item no two blocks in $Y$ can have the same first label in $P\setminus X$
        \item exactly~$j$ blocks of size~$\geq 2$ have second label~$\neg\square$
        \item exactly~$n-s-|X|$ blocks have second label~$\square$.
    \end{itemize}
\end{itemize}
The diagram~$D$ determines the tuple as follows (an example is shown in \autoref{fig:highcon}):
\begin{itemize}
    \item
    $P$ is the partition of~$\{-1,\dots , -n\}$ given by restricting the blocks of $D$ to the negative elements, \emph{i.e.}~to the nodes on the left hand side of the diagram.
    \item
    $X$ is the set of blocks in $P$ which, when viewed in~$D$, are connected to exactly one thing on the right (this can be a connection to the box, or to a single node).
    \item
    The number~$s$ is equal to the number of nodes on the right of $D$ not connected to a block in~$X$. 
    These nodes are precisely those which are singletons, or are connected to another element on the right, or to the box.
    Therefore, every node in one of the $j$ blocks of $D$ that have at least $2$ positive (right hand) nodes and are not connected to the box (as in \autoref{def-filtration}) is included and so $s\ge 2j$. 
    Also, none of the nodes that are connected to a block in $X$ are included, so $s \le n - |X|$. It follows that $n-|X|-s$ is the number of blocks connected to the box and to at least one node on the right.
    \item 
    $Y$ is the partition given by restricting~$D$ to the set of $s$ nodes on the right that are not connected to the blocks of~$X$ (we re-label these~$1,\dots, s$, maintaining the order). 
    \item
    The first entry of the labelling~$f$, in $\emptyset \cup P\setminus X$, indicates whether the blocks of~$Y$ are disconnected from the rest of~$D$ (in which case the label is~$\emptyset$), or connected to the left hand side (in which case the label is the block in $P\setminus X$ that they are connected to).
    Singletons in $Y$ cannot be connected to the left because otherwise they would be connected to a block in $X$ on the left. Thus their first label must be $\emptyset$.
    Two blocks in $Y$ cannot be connected to the same block in $P\setminus X$, so two first labels can only be the same if they are both $\emptyset$.
    
    The second entry of the labelling~$f$ is~$\square$ if the block in~$Y$ is connected to the box in~$D$ and~$\neg\square$ if it is not. The condition that there are exactly~$j$ blocks of size~$\geq 2$ with second label~$\neg\square$ accounts for the diagram being in the filtration quotient $F_jC_n/F_{j-1}C_n$. The condition that there are exactly $n-s-|X|$ blocks with second label $\square$ follows from the above observation that this is the number of blocks connected to the box, containing at least one node on the right.
\end{itemize}

The remaining data in the diagram determines an injective word with $s$ separators $\mathbf{a}$, of length $p+1-s$, on the set $X$, obtained as follows: If the $i$th node (from the top) on the right is connected to a block in~$X$, then the $i$th letter of $\mathbf{a}$ is the corresponding element of $X$. Otherwise the $i$th letter of~$\mathbf a$ is a separator, and there are exactly~$s$ of these.
\end{Def}

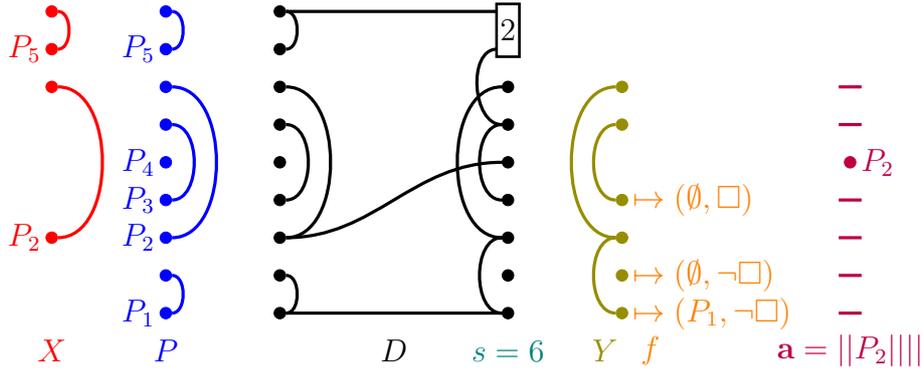
\begin{figure}[h!]
\centering
\[
    \begin{tikzpicture}[x=1.5cm,y=-.5cm,baseline=-2.05cm]
    
        \node[v] (a1) at (0,0) {};
        \node[v] (a2) at (0,1) {};
        \node[v] (a3) at (0,2) {};
        \node[v] (a4) at (0,3) {};
        \node[v] (a5) at (0,4) {};
        \node[v] (a6) at (0,5) {};
        \node[v] (a7) at (0,6) {};
        \node[v] (a8) at (0,7) {};
        \node[v] (a9) at (0,8) {};

        \node[v,blue] (c1) at (-1,0) {};
        \node[v,blue] (c2) at (-1,1) {}; 
        \node[v,blue] (c3) at (-1,2) {};
        \node[v,blue] (c4) at (-1,3) {}; 
        \node[v,blue] (c5) at (-1,4) {};
        \node[v,blue] (c6) at (-1,5) {}; 
        \node[v,blue] (c7) at (-1,6) {};
        \node[v,blue] (c8) at (-1,7) {}; 
        \node[v,blue] (c9) at (-1,8) {};

        \node[v,red] (d1) at (-2,0) {};
        \node[v,red] (d2) at (-2,1) {}; 
        \node[v,red] (d3) at (-2,2) {};
        \node[v,red] (d7) at (-2,6) {};
        
        \node (b1) at (2,0) {};
        \node (b2) at (2,1) {};
        \node[v] (b3) at (2,2) {};
        \node[v] (b4) at (2,3) {};
        \node[v] (b5) at (2,4) {};
        \node[v] (b6) at (2,5) {};
        \node[v] (b7) at (2,6) {};
        \node[v] (b8) at (2,7) {};
        \node[v] (b9) at (2,8) {};

        \node[v,olive] (e3) at (3,2) {};
        \node[v,olive] (e4) at (3,3) {};
        \node[v,olive] (e6) at (3,5) {};
        \node[v,olive] (e7) at (3,6) {};
        \node[v,olive] (e8) at (3,7) {};
        \node[v,olive] (e9) at (3,8) {};
        
        \draw[e] (a1) to[out=0, in=180] (b1);      
        \draw[e] (a7) to[out=0, in=180] (b5);
        \draw[e] (a1) to[out=0, in=0] (a2);
        \draw[e] (a8) to[out=0, in=0] (a9) to[out=0, in=180] (b9) to[out=180, in=180] (b7) to[out=180, in=180] (b3);
        \draw[e] (a3) to[out=0, in=0]   (a7);
        \draw[e] (a4) to[out=0, in=0]   (a6);
        \draw[e] (b2) to[out=180, in=180](b4) to[out=180, in=180] (b6);
        
        \draw[fill=white, line width=1] (1.9,-0.2) rectangle (2.1,1.2);
        \node at (2,0.5) {$2$};
        
        \draw[e,blue] (c1) to[out=0, in=0] (c2);
        \node[left,blue] at (c9) {$P_1$};
        \node[left,blue] at (c7) {$P_2$};
        \node[left,blue] at (c6) {$P_3$};
        \node[left,blue] at (c5) {$P_4$};
        \node[left,blue] at (c2) {$P_5$};
        \draw[e,blue] (c8) to[out=0, in=0] (c9);
        \draw[e,blue] (c3) to[out=0, in=0]   (c7);
        \draw[e,blue] (c4) to[out=0, in=0]   (c6);
        
        \draw[e,red] (d1) to[out=0, in=0] (d2);
        \node[left,red] at (d7) {$P_2$};
        \node[left,red] at (d2) {$P_5$};
        \draw[e,red] (d3) to[out=0, in=0]   (d7);
        
        \draw[e,olive] (e9) to[out=180, in=180] (e7) to[out=180, in=180] (e3);
        \draw[e,olive] (e4) to[out=180, in=180] (e6);
        \node[right,orange] at (e6) {$\mapsto(\emptyset, \square)$};
        \node[right,orange] at (e9) {$\mapsto(P_1, \neg\square)$};
        \node[right,orange] at (e8) {$\mapsto(\emptyset, \neg\square)$};

        \draw[e,purple] (4.9,2) to (5.1,2);
        \draw[e,purple] (4.9,3) to (5.1,3);
        \node[v,purple] at (5,4) {};
        \node[right,purple] at (5,4) {$P_2$};
        \draw[e,purple] (4.9,5) to (5.1,5);
        \draw[e,purple] (4.9,6) to (5.1,6);
        \draw[e,purple] (4.9,7) to (5.1,7);
        \draw[e,purple] (4.9,8) to (5.1,8);
        
        \node[red] at (-2,9) {$X$};
        \node[blue] at (-1,9) {$P$};
        \node at (1,9) {$D$};
        \node[teal] at (2,9) {$s=6$};
        \node[olive] at (2.85,9) {$Y$};
        \node[orange] at (3.25,9) {$f$};
        \node[purple] at (5,9) {$\mathbf{a}=||P_2||||$};
    \end{tikzpicture}
\]
\caption{This figure shows the process of extracting from the diagram~$D$ in \autoref{fig:example blob diagram} the tuple $(P, X, s, Y, f)$ and injective word $\mathbf{a}$ described in \autoref{def:5-tuple}. Blocks in~$P$, and the labeling $f$ of~$Y$ are indicated at their lowermost node.}
\label{fig:highcon}
\end{figure}

\begin{defn}
By the above discussion, we can define a map
\[\Phi_\ast:F_jC_n/F_{j-1}C_n\rightarrow \bigoplus_{P,X,s,Y,f}W_X^{(s)}[-s].\] 
The direct sum is over all $5$-tuples $(P,X,s,Y,f)$ 
satisfying the properties listed at the start of Definition~\ref{def:5-tuple}. 
A diagram~$D$ in $(F_jC_n/F_{j-1}C_n)_p$ is sent by $\Phi_p$ to the injective word with separators~$\mathbf{a}$ in the degree~$p$ part of the summand $W_X^{(s)}[-s]$ corresponding to ~$(P,X,s,Y,f)$, where $(P,X,s,Y,f)$ and $\mathbf{a}$ are obtained as in Definition~\ref{def:5-tuple}.
\end{defn}

We now prove that~$\Phi_\ast$ is a chain map and isomorphism. This will allow us to leverage the high connectivity of the complex of injective words with separators (\cite[Proposition 5.14]{BHP2021}) to a high connectivity result for~$C_n$, via the filtration.

\begin{lem}
    $\Phi_\ast$ is a chain map. 
    \begin{proof}
    First, we claim that the $5$-tuple $(P,X,s,Y, f)$ associated, via~$\Phi_\ast$,  to a basis diagram $D$ in $(F_jC_n/F_{j-1}C_n)_p$ is preserved in all diagrams appearing in the boundary of $D$.
    Recall from \autoref{def: complex for Pn} that the boundary map $\partial^p$ sends a diagram to the alternating sum of the diagrams obtained as follows: work through the nodes on the right of the diagram, and in each case move the node into the box.
    This clearly does not change the left-hand end of the diagram, and therefore all of the diagrams in the boundary have the same $X$ and $P$ associated to them.
    If the node that is moved into the box is a singleton, or was part of a block that was connected to the box, then these nodes are included in the count for~$s$, but after moving it into the box, the resulting diagram either has a singleton in the box or has a loop at the box, and therefore again vanishes. The other nodes counting towards $s$ are those that are part of a block with~$\geq 2$ elements from the right, and are not connected to the box. 
    There are exactly~$j$ such blocks, and so moving any of their nodes into the box gives zero in the quotient $(F_jC_n/F_{j-1}C_n)_p$. 
    Therefore the only nodes we can move into the box without getting zero, are those that are not counted by~$s$, \emph{i.e.}~$s$ remains constant under the boundary map. It follows that~$Y$ and $f$ remain constant, since~$Y$ partitions these~$s$ nodes and $f$ labels them.
    
    The above paragraph demonstrates that $F_jC_n/F_{j-1}C_n$ splits as a direct sum indexed by the 5-tuples $(P,X,s,Y,f)$.
    It now suffices to show that the assignment that sends a diagram with fixed $(P,X,s,Y,f)$ to the corresponding injective word with separators $\mathbf{a}$ respects the boundary map.
    But this is clear: moving a node joined to a block in~$X$ into the box corresponds exactly to deleting one of the non-separator letters from $\mathbf{a}$.
    \end{proof}
\end{lem}

\begin{lem}\label{lem:phi is an iso}
    $\Phi_\ast$ is an isomorphism.
    \begin{proof}
    We will prove that $\Phi_\ast$ is an isomorphism by showing that it is obtained from a bijection between the basis of  $(F_jC_n/F_{j-1}C_n)$, which is given by diagrams, and the basis of $\bigoplus_{X,P,s,Y,f}W_X^{(s)}[-s]$, which is given by injective words with separators.
    To do this, we will explain how to (re)build a diagram in $(F_jC_n/F_{j-1}C_n)$ from a tuple~$(P,X, s, Y,f)$ and an injective word with separators $\mathbf{a}$. 
    
    We work in degree $s+k-1$ in the summand~$W_X^{(s)}[-s]$ associated to a $5$-tuple $(P,X,s,Y,f)$. 
    We therefore take an injective word~$\mathbf{a}$ of length~$k$ with~$s$ separators, and we will build a diagram in~$(F_jC_n/F_{j-1}C_n)_{s+k-1}$. 
    We begin with an empty diagram with~$s+k$ nodes on the right hand side, and a box of size~$n-s-k$; this is possible since $s+k\leq s+|X|\leq n$, where the latter inequality is one of the conditions imposed on the $5$-tuple.
    Next, we build all the blocks on the left using~$P$, and draw half-edges from the blocks in $X$ to the right (don't connect these edges to anything yet). We place the injective word with separators vertically against the~$s+k$ nodes on the right hand side, and the word indicates connections from $k$ of the nodes to half-edges from $X$. We connect the remaining half edges from $X$ to the box. The separators indicate the positioning of the~$s$ nodes~$\{1,\dots, s\}$ which are then partitioned by $Y$, and labelled by~$f$. The first labels of~$Y$ indicate which blocks are connected to blocks on the left hand side in~$P\setminus X$. Finally, if the second label of a block in~$Y$ is~$\square$ we connect the block to the box.  
    Note that~$|X|-k$ blocks of~$X$ are connected to the box, and~$n-s-|X|$ blocks of $Y$ are connected to the box, the latter property being another of our conditions on the $5$-tuple. This means that exactly $n-s-k$ distinct blocks are connected to the box, and since this is the size of the box the diagram is non-zero in~$(C_n)_{s+k}$.
    The diagram lies in $F_jC_n/F_{j-1}C_n$ since exactly~$j$ blocks in~$Y$ of size~$\geq 2$ have second label~$\neg\square$ and are therefore not joined to the box, again by our conditions on the $5$-tuple.
    
    The last paragraph shows how to obtain, from a tuple $(P,X,s,Y,f)$ and an injective word $\mathbf{a}\in W_X^{(s)}[-s]_{s+k-1}$, a diagram in the basis of  $(F_jC_n/F_{j-1}C_n)_{s+k-1}$.
    It is now immediate to verify that this is inverse to the effect of $\Phi_\ast$ on bases, and this completes the proof.
    \end{proof}
\end{lem}

\begin{prop}\label{prop: filtration quotients are highly connected}
For all $0\leq j\leq \lfloor \frac{n}{2}\rfloor$, the filtration quotients~$F_jC_n/F_{j-1}C_n$ satisfy $H_i(F_jC_n/F_{j-1}C_n)=0$ for~$i\leq \frac{n-3}{2}$.
\begin{proof}
We first consider the case $n=0$, where the only possibility is that $j=0$ so that $F_0C_0 = C_0$.
The claim is then that $H_i(C_0)=0$ for $i\leq -3/2$, but since $C_0$ consists of a single copy of $\t$ in degree $-1$, this is immediate.

We now consider the case $n>0$.
Using \autoref{lem:phi is an iso} this is equivalent to the homology of $W_X^{(s)}[-s]$ vanishing in the desired range, for each $5$-tuple $(P,X,s,Y,f)$ satisfying the conditions of~\autoref{def:5-tuple}.
By~\cite[Proposition 5.14.]{BHP2021},~$H_i(W_X^{(s)})=0$ for~$i\leq |X|-2$, so that~$H_i(W_X^{(s)}[-s])=0$ for~$i\leq |X|+s-2$. 
It will therefore suffice to show that 
$\left\lfloor\frac{n-3}{2}\right\rfloor\leq |X|+s-2$,
or equivalently:
\[
    \begin{cases}
        n\leq 2|X| + 2s & \text{if $n$ even}\\
        n \le 2|X| +2s -1& \text{if $n$ odd.}
    \end{cases}
\]
Let us first prove that we always have $n \le 2|X| + 2s$.
Our conditions on the $5$-tuple $(P,X,s,Y,f)$ mean that $n-s-|X|$ is the number of blocks of $Y$ with second $f$-label $\square$, so that in particular $n-s-|X|\leq |Y|$. 
And since $Y$ is a partition of $\{1,\dots,s\}$ we have $|Y|\leq s$.
Combining the last two inequalities and rearranging gives us
$n\leq |X| + 2s$.
Because $|X|\ge 0$, we therefore have  $n\le 2|X| + 2s$. In particular, this proves the proposition if $n$ is even. If $n$ is odd, it certainly cannot be equal to $2|X|+2s$ which is even. Therefore it can be at most one smaller: $n\le 2|X|+2s-1$.
\end{proof}
\end{prop}

\begin{thm}\label{thm-acyclic}
$H_i(C_n)=0$ for~$i\leq \frac{n-3}{2}$.
\begin{proof}
By \autoref{prop: filtration quotients are highly connected}, the homology of the filtration quotient $F_jC_n/F_{j-1}C_n$ vanishes in degrees $i\leq \frac{n-3}{2}$ for all~$j$. The same then holds for $C_n$ itself by considering the long exact sequences associated to the short exact sequences $0\to F_{j-1}C_n\to F_jC_n\to F_jC_n/F_{j-1}C_n\to 0$.
\end{proof}
\end{thm}

\section{Proof of \autoref{thmb}}\label{sec-proof of thmb}

The proof of \autoref{thmb} directly mirrors the proof of~\cite[Theorem~B]{BHP2021}, with the following substitutions:
\begin{itemize}
    \item
    All instances of the Brauer algebra should be replaced with the partition algebra.
    \item
    The maps $\iota$ and $\pi$ of~\cite{BHP2021} should be replaced by the maps of the same name in the current paper.
    Similarly for the complex $C_\ast$.
    \item
    Theorem~5.4 of~\cite{BHP2021} should be replaced with~\autoref{thm-acyclic}.
    \item
    Theorem~4.1 of~\cite{BHP2021} should be replaced with~\autoref{thm-tor-quotient-Homology-Sm}.
\end{itemize}
We note that in the second paragraph of the proof of~\cite[Theorem~6.3]{BHP2021}, there is an error, and the words `odd' and `even' should be transposed.

\bibliographystyle{alpha}
\bibliography{repstab.bib}		

\end{document}